\documentclass[11pt]{article}

\usepackage{graphicx,latexsym,amsmath,amsfonts,amssymb,dsfont}
\usepackage{amsthm}
\usepackage{enumerate}
\usepackage{epsfig}
\usepackage{verbatim}
\usepackage{xcolor}
\usepackage{enumitem}

\usepackage{hyperref}
\usepackage{cite}
\usepackage[title,titletoc,toc]{appendix}

\usepackage[top=3.cm, bottom=4.0cm, left=2.2cm, right=2.2cm]{geometry}

\makeatletter \@addtoreset{equation}{section}
\renewcommand{\thesection}{\arabic{section}}
\renewcommand{\theequation}{\thesection.\arabic{equation}}

\newcommand{\E}{\mathds{E}}

\newcommand{\RR}{\mathbb{R}}
\newcommand{\NN}{\mathbb{N}}

\newtheorem{theorem}{Theorem}[section]
\newtheorem{lemma}{Lemma}[section]
\newtheorem{corollary}{Corollary}[section]
\newtheorem{definition}{Definition}[section]
\newtheorem{assumption}{Assumption}[section]
\newtheorem{example}{Example}[section]
\newtheorem{proposition}{Proposition}[section]
\newtheorem{remark}{Remark}[section]

\def\td{\textrm{d}}
\def\ex{\mathds{E}}
\def\prb{\mathds{P}}
\def\one{\mathds{1}}
\def\bx{\bar{X}}
\def\by{\bar{Y}}

\def\tr{\textrm{tr}}
\def\bb{\bar{b}}
\def\bsig{\bar{\sigma}}

\newcommand\num{\stepcounter{equation}\tag{\theequation}}

\def \D {\Delta}
\def \F {\mathcal{F}}

\def \ck {\check{\kappa}}

\def \lev {\left\|} \def \rev{\right\|}
\def \levv {\left|} \def \revv{\right|}
\def \lj {\left\langle} \def \rj {\right\rangle}
\def \lf {\left\lfloor} \def \rf {\right\rfloor}

\def\a{\alpha} \def\g{\gamma} 
   
 \def\k{\kappa}   
   \def\s{\sigma}

\def\be{\beta}

\def\le{\leqslant} \def\ge{\geqslant}

\def\1{\mathbf{1}}
\def\g{\gamma}

\def\tro{\tilde{\rho}}

\def\tx{\tilde{X}}

\title{$V$-Integrability, Asymptotic Stability And Comparison Theorem\\
		of Explicit Numerical Schemes for SDEs}

\author{{\L}ukasz Szpruch and X\={\i}l\'{\i}ng Zh\={a}ng
        \thanks{School of Mathematics, The University of Edinburgh, EH9 3FD, Edinburgh, UK (\texttt{l.szpruch@ed.ac.uk, xiling.zhang@ed.ac.uk}).}
        }

\date{}

\begin{document}
\maketitle

\thispagestyle{empty}

\begin{abstract}
\textsf{\em 
 Khasminski's  \cite{chas1980stochastic} showed that many of the asymptotic stability and the integrability properties of the solutions to the Stochastic Differential Equations (SDEs) can be obtained using Lyapunov functions techniques. These properties are rarely inherited by standard numerical integrators. In this article we introduce a family of explicit numerical approximations for the SDEs and derive conditions that allow to use Khasminski's  techniques in the context of numerical approximations,  particularly in the case where SDEs have non globally Lipschitz coefficients. 
 Consequently, we show that it is possible to construct a numerical scheme, that is bounded in expectation with respect to a Lyapunov function, and/or inherit the asymptotic stability property from the SDEs. Finally we show that using suitable schemes it is possible to recover comparison theorem for scalar SDEs.}



\end{abstract}

\section{Introduction}

The main goal of this article is to extend the applicability of Lyapunov function techniques to the numerical approximation of SDEs. This is achieved by modifying standard numerical schemes in such a way that, despite the lack of discrete-time It\^{o}'s formula, it is possible to mimic proofs developed for continuous time stochastic process. This then allows to analyse asymptotic and qualitative properties of the schemes using classical Khasminski's  techniques \cite{chas1980stochastic}. In particular, we investigate the integrability, asymptotic stability and convergence properties of numerical approximations for SDEs, paying particular attention to SDEs with non-globally Lipschitz drift and diffusion. This is a relatively new area of studies as the majority of research on numerical analysis for SDEs relies on, restrictively, global Lipschitz assumptions \cite{kloeden1992numerical,milstein2004stochastic}. If the global Lipschitz condition does not hold for either of the coefficients, Hutzenthaler, Jentzen and Kloeden \cite{hutzenthaler2011strong} showed that explicit Euler scheme has a fundamental flaw, and there is no hope for Lyapunov functions techniques to work in this setting. One of the consequences of this fact is that the classical Euler schemes fails to converge in strong sense. 

Let $(\Omega,\F,\prb)$ be a complete probability space with a right-continuous filtration $(\F_t)_{t\in I}$, where $I$ is a subinterval of $[0,\infty)$. Let $W_t$ be a $m$-dimensional $(\F_t)_{t\in I}$-adapted Wiener process. Consider SDE
\begin{equation} \label{eq:SDE}
\td X_t = b(t,X_t)\td t+\s(t,X_t)\td W_t,~t\in I,
\end{equation}
where $b:I\times\RR^d\rightarrow\RR^d$ and $\s=:I\times\RR^d\rightarrow\RR^{d\times m} $ are locally Lipschitz continuous. Consider Lyapunov function $V:\RR^d\to\RR^+$, which is at least twice differentiable. Both integrability and asymptotic stability properties of \eqref{eq:SDE} can be deduced by examining the diffusion operator 
\begin{equation*}
\mathcal{L}V(x)=\langle\nabla V(x),b(t,x)\rangle+\frac{1}{2}\tr\left[\s(t,x)V^{(2)}(x)\s(t,x)^\top\right],~\forall t\in I,~x\in\RR^d,
\end{equation*}
where $V^{(2)}(\cdot)$ is the Hessian matrix of $V(\cdot)$. When $I=[0,T]$ for $T>0$ fixed, standard results in stochastic analysis \cite{chas1980stochastic}, states that if there exists $\rho>0$ s.t 
\begin{equation}\label{LV1}
\mathcal{L}V(x) \le\rho V(x),~\forall t\in[0,T],~x\in\RR^d,
\end{equation}
then
\begin{equation}\label{eq:V_stab}
\E V(X_t) \le e^{\rho T} \E V(X_0),~\forall t\in[0,T].
\end{equation}

For the purpose of asymptotic stability of the equilibrium, one considers $I=[0,\infty)$ and the coefficients satisfying $b(t,0)\equiv0,~\s(t,0)\equiv0,~\forall t\ge0$ \footnote{ Given the well-posedness of the SDE \eqref{eq:SDE} one sees that the system has trivial solution (equilibrium) $X_t\equiv0,~\forall t\ge0$ a.s. when $X_0\equiv0$ a.s. } (see \cite{mao1991stability,mao2007stochastic}).
In this context one needs to consider a  Lyapunov function $V:\RR^d\to\RR^+\cup\{0\}$, which is at least twice differentiable, taking value $0$ at the origin, and strictly positive elsewhere (e.g. $V(\cdot)=|\cdot|^p$). Instead of \eqref{LV1}, a sufficient condition for  $X_t\to0$ a.s. as $t\to\infty$, regardless of the value of $X_0$ is 
\begin{equation}\label{LV2}
\mathcal{L}V(\cdot)\le-z(\cdot),
\end{equation}
for some $0\le z\in\mathcal{C}(\RR^d)$ such that  $z(\RR^d\setminus\{0\})>0$ and $z(0)=0$. Moreover if $z(\cdot)\ge\rho V(\cdot)$ for some constant $\rho>0$, then instead of \eqref{eq:V_stab} one has
\begin{equation}\label{V-exp}
\ex V(X_t)\le e^{-\rho t}\ex V(X_0)\to0,
\end{equation}
as $t\to\infty$, given $\ex V(X_0)<\infty$. Condition of the type \label{LV2} with $z(\cdot)\ge\rho V(\cdot)$ plays also a crucial role
in establishing ergodic property of the SDEs,  \cite{MR1931266}.

Majority of the research on integrability or stability of the numerical schemes 
relies on "simple" Lyapunov functions such as 
 $V(x)=|x|^{p},~p\ge1$, see for example \cite{kloeden1992numerical,milstein2004stochastic,hutzenthaler2012numerical,hutzenthaler2012strong,sabanis2013note},
 with the exception of \cite{Hutzenthaler2014,hutzenthaler2012numerical}.
Our aim is to handle more general cases, particularly with Lyapunov functions of the form
\begin{equation}\label{different_V}
V(x)=\sum_{i=1}^dc_ix_i^{p_i},\quad c_1,\cdots,c_d\in\RR,
\end{equation}
where $p_i$'s (non-negative) are not necessarily identical. This is necessary if one hopes to analyse many important SDEs in literature, see  \cite{hutzenthaler2012numerical,Hutzenthaler2014}\footnote{In \cite{hutzenthaler2012numerical,Hutzenthaler2014} authors investigated integrability, but not stability, properties for the explicit schemes allowing for Lyapunov functions of the form \eqref{different_V}  } and \autoref{van_der_pol} in the current paper.
It turns out that for a special class of Lyapunov functions $V=|x|^p,~p\ge 2$, drift-implicit Euler schemes allow to obtain discrete time counterpart of \eqref{eq:V_stab}, without global Lipschitz conditions, see \cite{higham2003strong,Szpruch2010monotone,szpruch-diss}. However, the need of solving an implicit equation at iteration of the algorithm might be costly. 
Furthermore, recently it has been demonstrated that even explicit Euler schemes, if appropriately modified, can have $L^p$-integrability property (and as by product strongly converge). Such modification of explicit schemes are sometimes conventionally called ``taming" - see \cite{hutzenthaler2012strong,hutzenthaler2012numerical,Hutzenthaler2014,tretyakov2012fundamental,sabanis2013note}.
This leads us to the approximation of \eqref{eq:SDE}  by 
\begin{equation}\label{eq:TE}
\bx_{k+1}=\bx_k+b^h(t_k,\bx_k)h+\s^h(t_k,\bx_k)\D W_{k+1},~k\in\NN,
\end{equation}
where $t_k=kh$ with $0<h\le 1$ being the step length of the uniform discretisation of $I$, and $\D W_{k+1}=W_{t_{k+1}}-W_{t_k}$. Usually the taming method $(b^h,\s^h)$ is chosen s.t. $b^h(t,x)\to b(t,x),~\s^h(t,x)\to\s(t,x)$ as $h\to0$ uniformly\footnote{Precise definition of these limits may vary.} in $(t,x)\in I\times\RR^d$. We also introduce the notation of the diffusion operator of the numerical scheme \eqref{eq:TE}
\begin{equation*}
\mathcal{L}^hV(x)=\langle\nabla V(x),b^h(t,x)\rangle+\frac{1}{2}\tr\left[\s^h(t,x)V^{(2)}(x)\s^h(t,x)^\top\right],~\forall t\in I,~x\in\RR^d.
\end{equation*}
The main challenge for taming $(b^h,\s^h)$ is to maintain the structure of the diffusion operator (that is preserve conditions \eqref{LV1} or \eqref{LV2}), while benefit from extra control on the growth of the coefficients.
Although integrability was established in literature for some specific explicit schemes that can be written in the form \eqref{eq:TE} condition \eqref{eq:V_stab} is rarely recovered. For example,  methodology developed in \cite{hutzenthaler2012numerical} generally does not ensure the recovery of \eqref{eq:V_stab}. Also in \cite{tretyakov2012fundamental} where similar argument was adopted for $V(x)=|x|^p,~p\ge2$ and \eqref{eq:V_stab} has not been exactly recovered, but with the right-hand-side term being some higher moment (higher than $p$) of $X_0$. Using $V(x)=|x|^p$ again, it is indeed recovered in \cite{Sabanis2014}, but the method generally does not allow strong convergence rate to be higher than $1/4$. We will show that the strong convergence rate $1/2$ will not be deteriorated, if one uses a projected scheme \eqref{eq:proj} proposed in this paper, and it allows the use of a general Lyapunov functions for \eqref{eq:V_stab}.

So far, the problem of asymptotic stability received less attention in literature than integrability. Nonetheless, considerable effort has been made in this direction (mainly using implicit scheme) in \cite{MR1931266,higham2001mean,higham2000stability,higham2003exponential,szpruch-diss,higham2008almost,MR2658159,MR3082312}. We will extend these results in two ways: a) we allow for the general Lyapunov functions; b) we use explicit Euler scheme. We need to highlight that to the best of our knowledge asymptotic stability for the explicit numerical schemes beyond Lipschitz setting has not been considered in literature so far. The mechanics behind it seems similar to that of the integrability  - the big difference, however, lies in the recovery of condition \eqref{LV2}. The issue is that for the scheme \eqref{eq:TE} we often can only obtain, that 
\begin{equation*}
\mathcal{L}V(\cdot)\le-z(\cdot) \implies  \mathcal{L}^hV(\cdot) \le - \rho^h(\cdot) z(\cdot), \quad \rho^h(\cdot) \ge 0,
\end{equation*}
and the lack of strict lower bound for $\rho^h(\cdot) z(\cdot)$ prevents from obtaining required stability results.  
One would face the same problem trying to recover ergodicity property of the underlying SDE using scheme \eqref{eq:TE}, (see \cite{MR1931266})
Nevertheless, the schemes of type \eqref{eq:TE} can recover almost-sure stability property (If $\rho^h(\RR^d\setminus\{0\})>0$ and $\rho^h(0)=0$), but $V$-exponential stability \eqref{V-exp} seems not to hold. The issue can be resolved by the aforementioned projected schemes,
\begin{equation} \label{eq:proj}
\bx_{k+1}=\Pi(\bx_k+b(t_k,\bx_k)h+\s(t_k,\bx_k)\D W_{k+1}),
\end{equation}
where $\Pi:\RR^d\to\RR^d$ is a projection function that can be customised. Projected schemes preserve almost-sure stability, by ensuring that  $ \rho^h(\cdot)>c>0$. It also enables $V$-exponential stability \eqref{V-exp} to be easily recovered.  
%

In the last part of the this article we will investigate the preservation of non-negativity and comparison theorem using explicit schemes. This is aimed at those SDEs whose solutions, for example, only stay in $[0,\infty)$. We will see that $b(t,0)\ge0,~\s(t,0)\equiv0$ is enough to guarantee $X_t\ge0$ a.s., but not necessarily the case for numerical schemes. We will show that simply by truncating the noise as is done in Section 1.3.4 in \cite{milstein2004stochastic}, one can easily recover non-negativity of the tamed Euler scheme. The same technique can almost immediately be extended to preserve comparison theorem on the SDE level, also in non-globally Lipschitz setting.

To summarize the main contributions of this paper: \vspace{-3mm}
\begin{itemize} [noitemsep,nolistsep]
\item We formulate general conditions for various explicit schemes under which it is possible to recover \eqref{eq:V_stab}, for a rich family of Lyapunov functions - e.g., \eqref{different_V}. 
\item  We investigate the asymptotic stability properties in non-globally Lipschitz setting using explicit schemes and general Lyapunov functions.
\item We propose a novel projected Euler scheme that recovers exponential $V$-stability properties, which allows optimal rate of strong convergence.
\item  We establish comparison theorem for Euler scheme, form which we can deduce non-negativity. Both properties hold true in non-Lipschitz setting.  
\end{itemize}

The article is structured as follows: In section \ref{sec:int} we present key 
result of this paper Theorem \ref{th:main} from which we deduce integrability, convergence and later on asymptotic stability of explicit Euler schemes. In section \ref{sec:stab}
we focus on asymptotic stability results paying particular attention to condition 
\eqref{V-exp}. Finally in section \ref{sec:comp} we present comparison theorem and  positivity preservation result for the explicit Euler schemes.

\section{$V$-Integrability of Tamed Euler Schemes} \label{sec:int}

In this section we investigate the integrability of tamed Euler schemes $\{\bx_k\}$, \eqref{eq:TE} or \eqref{eq:proj}, of the SDE
\begin{equation}\label{sde_integ}
\td X_t=b(t,X_t)\td t+\s(t,X_t)\td W_t,~t\in[0,T],
\end{equation}
for some $T>0$ fixed. Following \cite{hutzenthaler2012numerical}, let $p,d\in \NN^+$, $\gamma\in(0,1/p]$ and we considered the following space of Lyapunov functions $\mathcal{V}_\gamma^p\subset\mathcal{C}^{p+1}(\RR^d)$, where for $\mathbb{N}\ni p\geqslant2$ and $0<\gamma\leqslant\frac{1}{p}$,
\begin{equation}\label{V}
\mathcal{V}^p_\gamma:=\left\{V:~V(\cdot)\ge0,~\ker(V)=\{0\},~\exists c>0~\textrm{s.t.}~\|V^{(s)}(\cdot)\|_{\textrm{HS}}\leqslant c(1+V(\cdot))^{1-s\gamma},~\forall s\le p\right\}.
\end{equation}
Here $\|\cdot\|_{\textrm{HS}}$ denotes the Hilbert-Schmidt norm and $V^{(s)}$ denotes the $s$-th order derivative of $V$. For example, $V^{(1)}=\nabla V$ and $V^{(2)}$ is the Hessian matrix of $V$. Note that the space $\mathcal{V}^p_\gamma$ not only covers power functions $|\cdot|^p,~p>0$, but also covers polynomials of the form \eqref{different_V}. Hence it is rich enough for one to choose suitable Lyapunov functions for many of important SDEs (see \cite{hutzenthaler2012numerical} for more details).

\begin{remark}\label{remark1}
The function $|\cdot|^p$ for some even number $p$ is a candidate in
\begin{equation*}
\bar{\mathcal{V}}^p_{1/p}=\mathcal{V}^p_{1/p}\cap\left\{V:~V^{(p+1)}(\cdot)\equiv0,~\exists c>0~\textrm{s.t.}~\|V^{(s)}(\cdot)\|_{\textrm{HS}}\leqslant cV(\cdot)^{1-\frac{s}{p}},~\forall s\le p\right\}.
\end{equation*}
\end{remark}

Once we fix a Lypaunov function  $V\in\mathcal{V}^p_\gamma$ it will be useful if the growth conditions of the coefficients of the SDE \eqref{sde_integ} can be expressed in terms of $V$.
\begin{assumption} \label{as:poly}
For $V\in \mathcal{V}_\gamma^p$, $\exists K,\kappa>0$, s.t. 
\begin{align*}
|b(t,x)|\vee\|\s(t,x)\| \leqslant K\left(1+V(x)^{\kappa\gamma}\right), \quad \forall t\in[0,T],~x\in \RR^d.
\end{align*}
\end{assumption}
Take $V(\cdot)=|\cdot|^p\in\bar{\mathcal{V}}^p_{1/p}$, then \autoref{as:poly} essentially imposes polynomial growth condition on the coefficients of the SDE \eqref{sde_integ}. Indeed, we may observe that if there exists $L>0$ such that $\forall t,x,~\levv b(t,x)\revv \leqslant L(1+ \levv x \revv^{\kappa_1})$, one can find $K>0$ such that $\levv b(t,x)\revv \le K(1+V(x))^{\kappa_1/p}$. The same applies to the diffusion coefficient with polynomial growth of degree $\k_2$ and let $\k=\k_1\vee\k_2$. Expressing all estimates in terms of the chosen Lyapunov function \footnote{This corresponds to the Lyapunov-type functions $\tilde{V}(\cdot):=1+V(\cdot)$ defined in \cite{hutzenthaler2012numerical}.} makes all calculations convenient and transparent.

\begin{definition}
Let $V\in\mathcal{V}^p_\g$. The solution to the SDE \eqref{sde_integ} is $V$-integrable, if
\begin{equation*}
\sup_{t\in[0,T]}\ex V(X_t)<\infty.
\end{equation*}
A numerical scheme $\{\bx_k\}$ of the SDE \eqref{sde_integ} with step size $h$ is $V$-integrable, if
\begin{equation*}
\sup_{h>0}\max_{0\le k\le\lf T/h\rf}\ex V(\bx_k)<\infty.
\end{equation*}
\end{definition}

To clarify the idea of this section without going into too much technical details let us consider a motivational example.
\begin{example}
Let $(X_t)_{0\le t\le T}$ be the solution to the following $1$-d autonomous SDE 
\begin{align} \label{eq:SDE_ex}
\td X_t = b(X_t)\td t + \s(X_t)\td W_t,
\end{align}
where $b$ and $\s$ satisfy \autoref{as:poly} and monotonicity condition:
\begin{align} \label{eq:con_ex}
2 xb(x) + |\s(x)|^{2}\le \rho(1+|x|^2) \quad \forall x\in \RR.
\end{align}
\end{example}
Note that \eqref{eq:con_ex} corresponds to the special case of the Lyapunov function $V(x)=|x|^2\in\hat{\mathcal{V}}_{1/2}^2$, and it immediately follows that $\forall t\ge0$,
\[
\ex V(X_t)<\ex(1+V(X_t))\le e^{\rho t}\ex(1+V(X_0)).
\]
We are seeking some condition under which the tamed Euler scheme
\[
\bx_{k+1} = \bx_{k} + b^{h}(\bx_k)h + \s^{h}(\bx_{k})\D W_{k+1},
\]
is also $V$-integrable (bounded second moments in this case). Let us first square both sides of the scheme to get
\begin{equation}\label{square}
\ex_k|\bx_{k+1}|^2 = |\bx_{k}|^2 + \big(2\bx_{k}b^{h}(\bx_{k}) + |\s^{h}(\bx_{k})|^2 \big)h  +  |b^{h}(\bx_{k})|^2h^2,
\end{equation}
where $\E_{k}(\cdot):=\E(\cdot|\F_{k})$. If a taming method is chosen such that $\exists\mu>0$,
\begin{align} \label{eq:tame1_ex}
|b^{h}(x)|^2h \le \mu(1+V(x)), \quad \forall x\in \RR,
\end{align}
and
\begin{align} \label{eq:tame2_ex}
2xb^{h}(x) + |\s^{h}(x)|^2\le \rho(1+V(x)), \quad \forall x\in \RR,
\end{align}
then $\forall1\le k\le\lf T/h\rf$,
\begin{align*}
\ex_k(1+V(\bx_{k+1}))\le&1+V(\bx_{k}) + (\rho +\mu)(1+V(\bx_k))h\\
\implies\ex V(\bx_{\lf T/h\rf})\le&e^{(\rho+\mu)T}\ex(1+V(X_0)).
\end{align*}
One can use taming method, e.g.,
\begin{equation} \label{eq:tame_ex}
b^{h}(t,x) := \frac{b(t,x)}{1+G_b(x,h)}, 
\quad 
\s^{h}(t,x) := \frac{\s(t,x)}{1+G_\s(x,h)}, \quad \forall t\in[0,T],~x\in \RR^d,
\end{equation}
for some $G_b(\cdot,\cdot),G_\s(\cdot,\cdot)\ge0$. Then condition \eqref{eq:tame2_ex} holds provided $1+ G_b(x,h) \le (1 + G_\s(x,h))^2$. Furthermore for \eqref{eq:tame1_ex} we take  $G_\s(x,h)= G_b(x,h):= CV(x)^{\k_0/2}h^{\be} $, with $C=K/\sqrt{\mu}$,
$k_0=(\kappa-1)_{+}$ and $\be=1/2$, so that
\begin{align*}
|b^h(x)|h^{1/2} = \frac{|b(x)|h^{1/2}}{1 + CV(x)^{\k_0/2}h^{1/2}} 
\le \frac{KV(x)^{\k/2}h^{1/2}}{1 + CV(x)^{\k_0/2}h^{1/2}}  
\le \sqrt{\mu} V(x)^{1/2},
\end{align*}
as required.

\subsection{Taming Conditions for $V$-Integrability}

The difficulty with exploring powerfulness of Lyapunov technique in the context of numerical schemes for SDEs is lack of discrete time It\^{o} formula. However suitable tamed schemes allows to recover many of the classical results by appropriately controlling a remainder terms of Taylor expansions.  This is a topic of Theorem \ref{th:main}.  

In the first part of this section we focus on the subspace of $\mathcal{V}^p_\gamma$ denoted by $\hat{\mathcal{V}}_\gamma^p=\mathcal{V}_{\g}^p\cap\{\nabla^{p+1}V\equiv0\}$ (this class contains almost all examples of polynomial Lyapunov functions presented in \cite{hutzenthaler2012numerical}). As an example one may consider a very popular Lypaunov function $V(x)=|x|^p,~p\ge2$, which allows us to explore so called one-sided Lipschitz property of the drift coefficient of the SDE \eqref{eq:SDE}. Later on we will show that integrability result can be extended to the whole space $\mathcal{V}_\g^p$. 

\begin{theorem} \label{th:main}
Suppose for the tamed coefficients in \eqref{eq:TE} there is a Lyapunov function $V \in \hat{\mathcal{V}}_\g^p,~p\ge2$ s.t. $\ex V(X_0)<\infty$ and
\begin{align} \label{eq:con_L}
\mathcal{L}^hV(x)\le \rho(1+V(x)),~\forall x\in\RR^d,
\end{align}
for some $\rho>0$. Also assume that one can find such tamed coefficients that $\exists\mu>0$ s.t.
\begin{equation} \label{eq:b_coeff}
\left|b^{h}(t,x)\right| h^{1/2}\vee\lev \s^h(t,x)\rev h^{1/4}  \leqslant \mu(1+V(x))^\g.
\end{equation}
Then there exists a constant $\tro:=\tro(\mu)$ s.t. 
\begin{equation*}
\E V(\bx_k)\leqslant e^{(\rho+\tro)T}\E(1+V(X_0))<\infty,\quad\forall0\le k\leqslant\lf T/h\rf.
\end{equation*}
\end{theorem}
\begin{proof}
Since $V \in \hat{\mathcal{V}}_\g^p$, one has a finite Taylor expansion for
\begin{align} \label{eq:l0}
\E_k(1+V(\bx_{k+1})) = & 1 + V(\bx_k) + \E_k\sum_{|\a|\le p} \frac{\partial^{\a}V(\bx_{k})}{\a!}(\bx_{k+1}-\bx_{k})^\a.
\end{align}
For the convenience of notation denote $\bb_k:=b^h(t_k,\bx_k),~\bsig_k:=\sigma^h(t_k,\bx_k)$. Taking $1\le s \le p$ and by multi-index notation we have,
 \begin{align*}
& \E_{k}\sum_{|\a|=s}\frac{\partial^\a V(\bx_{k})}{\a! }(\bx_{k+1}-\bx_{k})^\a\\ 
&= \E_{k}\sum_{|\a|=s}\frac{\partial^\a V(\bx_{k})}{\a! }
\sum_{\nu \le \a} \binom{\a}{\nu} (\bb_kh)^{\a-\nu}(\bsig_k\D W_{k+1})^{\nu}\num\label{eq:l1}, 
\end{align*} 
which by multinomial theorem  is equivalent to
\begin{align*}
& \E_{k}\sum_{|\a|=s}\frac{\partial^\a V(\bx_{k})}{\a! }(\bx_{k+1}-\bx_{k})^\a\\ 
& = \frac{1}{s!}\E_{k}\sum_{\be_1+\ldots+\be_d=s}\binom{s}{\be_1,\ldots,\be_d}\Big(\left(\bx_{k+1}^{(1)}-\bx_{k}^{(1)}\right)^{\be_1}
 \cdots\left(\bx_{k+1}^{(d)}-\bx_{k}^{(d)}\right)^{\be_d}\Big) 
\frac{\partial^s}{\partial x_1^{\be_1}\ldots\partial x_d^{\be_d}}V(\bx_{k}),
 \end{align*}
where, for $i=1,\ldots,d$, 
\begin{align*}
&\left(\bx_{k+1}^{(i)}-\bx_{k}^{(i)}\right)^{\be_i} = \left(\bb_k^{(i)} + \sum_{j=1}^{m}\bsig_k^{(ij)}\D W_{k+1}^{(j)}\right)^{\be_i} \\
&= \sum_{\delta^{i}_0+\ldots+\delta^{i}_m=\be^{i}}\binom{\be_i}{\delta^{i}_0,\ldots,\delta^{i}_m}
(\bb_k^{(i)})^{\delta^{i}_0}(\bsig_k^{(i1)})^{\delta^{i}_1}\cdots(\bsig_k^{(im)})^{\delta^{i}_m}(\D W_{k+1}^{(1)})^{\delta^{i}_1}\cdots(\D W_{k+1}^{(m)})^{\delta^{i}_m}.
\end{align*}
Therefore, it is easy to see that the first two terms of \eqref{eq:l1} are of the form 
\begin{align*}
 \E_{k}\sum_{|\a|=1}\frac{\partial^\a V(\bx_{k})}{\a! }(\bb_kh + \bsig_k\D W_{k+1})^\a = \lj\bb_k,\nabla V(\bx_k)\rj h,
\end{align*}
and
\begin{align*}
&\E_{k}\sum_{|\a|=2}\frac{\partial^\a V(\bx_{k})}{\a! }(\bb_kh + \bsig_k\D W_{k+1})^\a = \frac{1}{2}\sum_{s=1}^{m} \lj\bsig_k^{(s)},\nabla^2V(\bx_{k})\bsig_k^{(s)}\rj h + \frac{1}{2} \lj \bb_k,\nabla^2V(\bx_{k})\bb_k\rj h^2 \\
& =  \frac{1}{2}\sum_{i,j=1}^{d}\sum_{s=1}^{m} \frac{\partial^2 V}{\partial x_i \partial x_j}(\bx_{k}) \bsig_k^{(is)}\bsig_k^{(js)}h + \frac{1}{2}\sum_{i,j=1}^{d} \frac{\partial^2 V}{\partial x_i \partial x_j}(\bx_k) \bb_k^{(i)}\bb_k^{(j)}h^2\\
&\le\frac{1}{2}\tr\left[V^{(2)}(\bx_k)\bsig_k\bsig_k^\top\right]h+\frac{1}{2}\lev V^{(2)}(\bx_k) \rev\levv\bb_k\revv h^2.
\end{align*}
We can now analyse the rest of the expansion \eqref{eq:l0} with $|\a|\ge3$. Recall that
\begin{equation*}
\left(\bx^{(i)}_{k+1}-\bx^{(i)}_{k}\right)^s = \left(\bb_k^{(i)} + \sum_{j=1}^{m}\bsig_k^{(ij)}\D W_{k+1}^{(j)}\right)^{s} \le \sum_{r=0}^s\binom{s}{r}\left(\bb_k^{(i)}h\right)^{s-r}\left(\bsig_k^{(i,\cdot)}\D W_{k+1}\right)^r.
\end{equation*}
Notice that, due to marginality property of Brownian motion, the terms with odd $r$'s in the above summation are zero under conditional expectation $\ex_k$. Hence, with a bit relabelling,
\begin{align*}
\E_{k}&\sum_{|\a|=s}\frac{\partial^\a V(\bx_{k})}{\a! }(\bx_{k+1}-\bx_{k})^\a \le \lev V^{(s)}(\bx_k) \rev \frac{d^{s-1}}{s!}\sum_{r=0}^{\lfloor s/2\rfloor}\binom{s}{2r}\levv\bb_k\revv^{s-2r}\lev\bsig_k\rev^{2r}h^{s-r}\\
\le&\phi_s\lev V^{(s)}(\bx_k) \rev\sum_{r=0}^{\lfloor s/2\rfloor}\levv\bb_k\revv^{s-2r}\lev\bsig_k\rev^{2r}h^{s-r},
\end{align*} 
where $0<\phi_s\le d^{s-1}\binom{s}{r}/s! \le d^{s-1}/(\lfloor s/2\rfloor!)^2$. Returning to \eqref{eq:l0} and using the above estimates, we obtain
\begin{equation}\label{taylor}
\E_{k}(1+V(\bx_{k+1})) = 1+V(\bx_{k}) + \mathcal{L}^hV(\bx_k)h + R^hV(\bx_k),
\end{equation}
where, by relabelling the indices in the summation,
\begin{equation}\label{Rh}
R^hV(\bx_k)\le \frac{1}{2}\lev V^{(2)}(\bx)\rev\levv\bb_k\revv^2h^2+\sum_{\substack{3\le i+2j\le p\\i,j\in\NN}}\phi_{i+2j} \lev V^{(i+2j)}(\bx_k) \rev\levv \bb_k\revv^i\lev\bsig_k\rev^{2j}h^{i+j}.
\end{equation}
Now given that
\begin{align*}
\levv b^{h}(t,x)\revv h^{1/2}\vee\lev\s^h(t,x)\rev h^{1/4} \le \mu(1+V(x))^\gamma,
\end{align*}
and that $\exists c>0$ for all $x \in\RR^d$  such that for $\forall1\leqslant s\leqslant p$, 
\[
\lev V^{(s)}(x)\rev \le c(1+V(x))^{1-s\gamma},
\]
we have
\begin{align*}
R^hV(\bx_k) \le & \frac{1}{2}c\mu^2(1+V(\bx_k))h+\sum_{\substack{3\le i+2j\le p\\i,j\in\NN}}\phi_{i+2j}c\mu^{i+2j}(1+V(\bx_k))h^\frac{i+j}{2}\\
\le & \left(\frac{1}{2}c\mu^2+\sum_{s=3}^p\phi_sc\mu^s\right)(1+V(\bx_k))h
\end{align*}
Choose $\tro:=c\mu^2/2+c\sum_{s=3}^p\phi_s\mu^s$ and the theorem is proved.
\end{proof}

\begin{remark}
For $p=2$ one only needs to check condition \eqref{eq:b_coeff} for $b^h(\cdot,\cdot)$.
\end{remark}

\begin{remark} \label{rmk:const}
For the practical implementation we can take $\mu\le 1$ in \eqref{eq:b_coeff} and since
$
\sup_{3 \le s \le p }\phi_s \le d^{p-1},
$
and we immediately can choose
$
\tro:= c(p-1)d^{p-1}\mu^2.
$
Therefore by a suitable choice of parameter $\mu,~\tro$ can be arbitrarily small. 
\end{remark}

In the similar way we extend applicability of tamed Euler schemes to Lyapunov-type functions $V \in \mathcal{V}^p_\gamma$. It turns out that the smoothness of $V$ affects the rate of taming of the diffusion coefficient. 
\begin{proposition} \label{lem:main}
Let $V \in \mathcal{V}^p_\gamma,~p\ge2$. Suppose $\exists\rho>0$ s.t. $\mathcal{L}^hV(\cdot)\le\rho V(\cdot)$, and $\exists\mu>0$ s.t. 
\begin{equation}
\levv b^{h}(t,x)\revv h^{\be_1}\vee\lev \s^h(t,x)\rev h^{\be_2} \le \mu(1+V(x))^\g,~\forall t,x,
\end{equation}
where $\be_1\le 1/2$ and $\be_2\le 1/2-1/(p\wedge4)$. 
Then $\exists\tro:=\tro(\mu)$ s.t. 
\[
\E V(\bx_k) \le e^{(\rho + \tro)T} \E(1+V(X_{0})),~\forall 0\le k\le\lf T/h\rf.
\]
\end{proposition}
\begin{proof}
The proof is very similar to the proof of \autoref{th:main}. We write
\begin{align} \label{eq:expansion_tr}
\ex_k(1+V(\bx_{k+1}))= &1+V(\bx_k) + \sum_{|a|\leq p-1} \frac{ \partial^{\a}V(\bx_k)}{\a!}(\bx_{k+1}-\bx_k)^{\a} \nonumber \\
& +  p \sum_{|\a|=p}\frac{(\bx_{k+1}-\bx_{k})^{\a}}{\a!}\int_{0}^{1}
(1-t)^{p-1}\partial^{\a}V(\bx_k+t(\bx_{k+1}-\bx_k))\td t. 
\end{align}
It is therefore enough to look into a remainder term for $p\ge2$
\begin{align*}
\tilde{R}^h:= & p\sum_{|\a|=p}\frac{(\bx_{k+1}-\bx_k)^{\a}}{\a!}\int_{0}^{1}
(1-t)^{p-1}\partial^{\a}V(\bx_k+t(\bx_{k+1}-\bx_k))\td t \\
\le & p\sum_{|\a|=p}\frac{|(\bx_{k+1}-\bx_k)^{\a}|}{\a!}\int_{0}^{1}
(1-t)^{p-1} \lev V^{(p)}(\bx_k+t(\bx_{k+1}-\bx_k)) \rev \td t  \\
\le & cp\sum_{|\a|=p}\frac{|(\bx_{k+1}-\bx_k)^{\a}|}{\a!}
\int_{0}^{1} (1-t)^{p-1}\left(1+V(\bx_k+t(\bx_{k+1}-\bx_k))\right)^{1-p\gamma} \td t.
\end{align*}
By Lemma 2.12 in \cite{hutzenthaler2012numerical} we have
\[
1+V(x+y)\le c^\frac{1}{\g}2^{\frac{1}{\g}-1}\Big(1+V(x) + \levv y\revv^\frac{1}{\g} \Big),~\forall x,y\in\RR^d,
\]
which leads to 
\begin{align*}
\Big(1+V(x+y)\Big)^{1-p\g} \le & c^{\frac{1}{\g}-p}2^{\big(\frac{1}{\gamma}-p\big)(1-\gamma)}\Big(1+V(x) + \levv y\revv^\frac{1}{\g} \Big)^{1-p\g} \\
 \le & (2c)^{\frac{1}{\g}-p}\Big((1+V(x))^{1-p\g} + \levv y\revv^{\frac{1}{\g}-p} \Big),
\end{align*}
for $\g\in(0,1/p]$. Consequently 
\begin{align*}
\tilde{R}^h \le & cp \sum_{|\a|=p}\frac{|(\bx_{k+1}-\bx_k)^{\a}|}{\a!}
(2c)^{\frac{1}{\g}-p}\Big((1+V(\bx_k))^{1-p\g} + \levv \bx_{k+1} - \bx_k \revv^{\frac{1}{\g}-p} \Big) \\
\le & p\frac{c^{\frac{1}{\g}-p+1}2^{\frac{1}{\g}-p}}{p!}\left(\sum_{i=1}^d |\bx^{(i)}_{k+1} - \bx^{(i)}_{k}|\right)^p \Big((1+V(\bx_k))^{1-p\g} + \levv \bx_{k+1} - \bx_k \revv^{\frac{1}{\g}-p} \Big) \\
\le & \frac{d^{p-1}  c^{\frac{1}{\g}-p+1 }2^{\frac{1}{\g}-p}}{(p-1)!}\levv \bx_{k+1} - \bx_k \revv^{p} \Big((1+V(\bx_k))^{1-p\g} + \levv \bx_{k+1} - \bx_k \revv^{\frac{1}{\g}-p} \Big)\\
\le & c \tilde{\psi}\left(\levv \bb_k \revv^ph^p + \lev \bsig_k \rev^ph^{p/2}\right)(1+V(\bx_k))^{1-p\g} +c\tilde{\psi} \big( \levv \bb_k \revv^\frac{1}{\g}h^\frac{1}{\g} + \lev \bsig_k \rev^\frac{1}{\g}h^\frac{1}{2\g} \big),
\end{align*}
where by the similar argument as in \autoref{rmk:const},
\[
\tilde{\psi} = \frac{(d(m+1))^{\frac{1}{\gamma}-1} (2c)^{\frac{1}{\g}-p}}{(p-1)!}.
\]
Now given that $\exists\beta_{1,2},\mu>0$ s.t.
\begin{align*}
\levv b^{h}(t,x)\revv h^{\be_1} \le \mu(1+V(x))^\g~\text{and}~\lev \s_k^h(t,x)\rev h^{\be_2}  \le \mu(1+V(x))^\g,
\end{align*}
then $\exists\tro=\tro(\mu)>0$ s.t.
\[
R^hV(\bx_k) \le \tro (1+V(\bx_{k})) h,
\]
as in \eqref{taylor}, where for $p\geqslant2$,
\begin{align*}
R^h&V(\bx_k) \le \frac{1}{2}\lev V^{(2)}(\bx)\rev\levv\bb_k\revv^2h^2+\sum_{\substack{3\le i+2j\le p-1\\i,j\in\NN}}\phi_{ij} \lev V^{(i+2j)}(\bx_k) \rev\levv \bb_k\revv^i\lev\bsig_k\rev^{2j}h^{i+j} + \tilde{R}^h\\
\le & \left(\frac{1}{2}c\mu^2h^{1-2\be_1}+\sum_{\substack{3\le i+2j\le p-1\\i,j\in\NN}}\phi_{ij}c\mu^{i+2j}h^{(1/2-\be_1)i+(1/2-2\be_2)j}\right)(1+V(\bx_k))h \\
& + c\mu^p\tilde{\psi}(1+V(\bx_k))\Big(h^{p(1-\beta_1)-1}+h^{p(1/2-\beta_2)-1}\Big)h+c\mu^\frac{1}{\g}\tilde{\psi}(1+V(\bx_k))\left(h^{\frac{1-\beta_1}{\g}-1}+h^{\frac{1-2\beta_2}{2\g}-1}\right)h\\
\le & \tro (1+V(X_k)) h,
\end{align*}
for $\be_1\le1/2$ and $\be_2\le1/2-1/(p\wedge4)$, and
\begin{equation*}
\tro := \frac{1}{2}c\mu^2+c\sum_{s=3}^{p-1}\mu^s\phi_s+2c\mu^p\tilde{\psi},
\end{equation*}
where $\{\phi_s\}$ are as in \autoref{rmk:const}.
\end{proof}


\subsection{Taming Choices}\label{sec:taming_choices}

The results in the previous section give us general conditions for the tamed Euler scheme \eqref{eq:TE}. A natural question would be if the assumptions in \autoref{th:main} and \autoref{lem:main} can be satisfied with specific taming methods, i.e., for $V\in\mathcal{V}^p_\g$ whether
\begin{equation}\label{eq:Lyp_ex}
\mathcal{L}V(x)\le\rho(1+V(x))\implies\mathcal{L}^hV(x)\le\bar{\rho}(1+V(x)),~\forall x\in\RR^d,
\end{equation}
for some $\rho,\bar{\rho}>0$, and
\begin{equation} \label{eq:con_tame}
|b^{h}(t,x)| h^{\be_1} \vee \|\s^h(t,x)\| h^{\be_2} \le \mu(1+V(x))^\g,~\forall t\in[0,T],~x\in\RR^d,
\end{equation}
for some $\be_1\le1/2$ and $\be_2\le1/2-1/(p\wedge4)$ hold.

\paragraph{Balanced Schemes}Let us first look at the balanced schemes proposed in \cite{hutzenthaler2012strong,tretyakov2012fundamental,Sabanis2014}, which in general is of the form
\begin{equation}\label{scm11}
b^{h}(t,x) := \frac{b(t,x)}{1+G_b(x,h)},~\s^{h}(t,x) := \frac{\s(t,x)}{1+G_\s(x,h)},~\forall t,x,
\end{equation}
where $0\leq G_b(\cdot,h),G_\s(\cdot,h) \to 0 $ as $h \to 0$. In this case requirement \eqref{eq:Lyp_ex} is interpreted as
\begin{align*}
\mathcal{L}^hV(x) &:=\nabla V(x)\cdot b^h(t,x)+\frac{1}{2}\tr\left[V^{(2)}(x)\s^h(\s^h)^\top(t,x)\right]\\
& = \frac{\nabla V(x)\cdot b(t,x)}{1+ G_b(x,h)}
+\frac{1}{2}\frac{\tr\left[V^{(2)}(x)\s\s^\top(t,x)\right]}{(1+ G_\s(x,h))^2} \le \rho (1+V(x)).
\end{align*}
Hence, condition \ref{eq:Lyp_ex} holds if either of the following conditions is satisfied:
\begin{itemize}\vspace*{-5mm}
\item[i)]$1+ G_b(x,h) = ( 1+ G_\s(x,h))^2,~\forall x,h$;
\item[ii)]$1+ G_b(x,h) \le ( 1+ G_\s(x,h))^2,~\forall x,h$, if  $\tr\left[V^{(2)}(x)\s\s^\top(t,x)\right]>0,~\forall x\in \RR^d$ (which will be the case for most Lyapunov functions).
\end{itemize}\vspace*{-5mm}
One may consider case i) and let, e.g.,
\begin{align*}
G_b(x,h):= 2CV(x)^{\k^\ast\g}h^{\be_2} + C^2V(x)^{2\k^\ast\g}h^{2\be_2}
\quad \text{and} \quad G_\s(x,h):= CV(x)^{\k^\ast\g}h^{\be_2}.
\end{align*} 
In order for \eqref{eq:con_tame} to hold we take $\be_1=2\be_2$, $C\ge K/\mu$ and $k^\ast\ge\k-1$ so that
\begin{align*}
\|\s^h(t,x)\| h^{\be_2} = \frac{\|\s(t,x)\| h^{\be_2}}{1+CV(x)^{\k^\ast\g}h^{\be_2}}
\le \frac{K(1+V(x))^{\k\g} h^{\be_2}}{1+CV(x)^{\k^\ast\g}h^{\be_2}} \le\mu (1+V(x))^\g,
\end{align*}
by \autoref{as:poly}. We also need to choose $C^2\ge K/\mu$ so that 
\begin{align*}
|b^{h}(t,x)| h^{\be_1} \le \frac{K(1+V(x))^{\k\g} h^{2\be_2}}{1+ 2CV(x)^{\k^\ast\g}h^{\be_2} + 
C^2V(x)^{2\k^\ast\g}h^{2\be_2}}\le \mu (1+V(x))^\g,
\end{align*}
as $2\k^\ast\ge\k-1$. Therefore we choose $\k^\ast\ge\k-1$ and $C\ge (K/\mu)\vee1$, which gives a reasonable taming method for the scheme to be bounded with respect to $V$.

\paragraph{Projected Schemes} Here we propose a new projected Euler scheme: 
\begin{equation}\label{scm21}
\bx_{k+1}=\Pi\left(\bx_k+b(t_k,\bx_k)h+\sigma(t_k,\bx_k)\Delta W_{k+1}\right),
\end{equation}
where $\Pi:\mathbb{R}^d\to\mathbb{R}^d$ defined s.t. $|\Pi x|=|x|\wedge h^{-r}$ with $r>0$ to be chosen. For example one can define $\Pi x=\left(\Pi_ix_i\right)_{i=1}^d$ as a truncation, where $\Pi_ix_i=(-h^{-r}\vee x_i\wedge h^{-r})/\sqrt{d}$, or $\Pi(x)=\min\{1,h^{-r}|x|^{-1}\}x$.  See also \cite{Jean-FrancoisChassagneux2014} for different types of projected Euler scheme for scalar SDEs. In order to ensure $|\bx_k|\le h^{-r}$ for all $k\ge0$ we may assume $|X_0|\le h^{-r}$, otherwise send in $\Pi X_0$ for the first iteration. Integrability of this scheme becomes very straightforward if a little bit more assumption is imposed on the Lyapunov function. Later on in the stability section we will show that this scheme gives a solution to the preservation of $V$-exponential stability, which balanced schemes fail to achieve.
\begin{theorem}\label{project+tame}
Let \autoref{as:poly} hold and $V\in\mathcal{V}^p_\g$ s.t. $V(x)\le V(y)\le\nu(1+|y|^q)$ for some $\nu>0$ and all $x,y\in\RR^d$ s.t. $|x|\le|y|$. Suppose $\ex V(X_0)<\infty$, and
\begin{equation*}
\mathcal{L}V(x)\le\rho(1+V(x)),~\forall x\in\RR^d,
\end{equation*}
for some $\rho>0$. Then, for $r\le\be_2/((\k-1)q\g)$, the projected scheme \eqref{scm21} is $V$-integrable, where $\be_2$ is as in \autoref{lem:main}.
\end{theorem}
\begin{proof}
The same arguments adopted in the proofs of \autoref{th:main} and \autoref{lem:main} imply
\begin{align*}
V(\bx_{k+1})=&V\left(\Pi(\bx_k+b(t_k,\bx_k)h+\sigma(t_k,\bx_k)\Delta W_{k+1})\right)\\
\leqslant&V(\bx_k+b(t_k,\bx_k)h+\sigma(t_k,\bx_k)\Delta W_{k+1})\\
\leqslant&V(\bx_k)+\mathcal{L}V(\bx_k)h+R^hV(\bx_k)+M_{k+1},\num\label{rem1}
\end{align*}
where $M_{k+1}$ is a local martingale, as the expression given in \eqref{taylor}. This immediately shows that one only needs to work with $\mathcal{L}V(x),~b(t,x)$ and $\s(t,x)$ directly for $|x|\le h^{-r}$. Thus \eqref{eq:Lyp_ex} is redundant and
\begin{align}
|b(t,x)|h^{\be_1}\vee\|\s(t,x)\|h^{\be_2}\le&K(1+V(x))^{\k\g}h^{\be_2}\le2K\nu\left(1+|x|^{q(\k-1)\g}\right)(1+V(x))^\g h^{\be_2}\nonumber\\
\le&4K\nu h^{\be_2-r(\k-1)q\g}(1+V(x))^\g=:\mu(1+V(x))^\g,\label{trun_bound}
\end{align}
by choosing $r\le\be_2/((\k-1)q\g)$, which achieves \eqref{eq:con_tame}. The result thus follows by \autoref{lem:main}.
\end{proof}

\paragraph{Strong Convergence} Now given the integrability (bounded moments) of the scheme we can explain how in general one may establish strong convergence of the tamed Euler scheme \eqref{eq:TE}. We build up the strong convergence of tamed Euler schemes on the results established in \cite{hutzenthaler2012numerical} (see their Definition 3.1 and Corollary 3.12) and \cite{tretyakov2012fundamental}  (see the proof of Lemma 3.2 and Theorem 2.1). Roughly, both results state that provided that appropriate moment bounds (corresponding to $V(\cdot)=|\cdot|^p$) for the tamed Euler \eqref{eq:TE} are achieved, and that the strong and weak one-step differences against the standard Euler scheme are given by appropriate rates, then the tamed Euler scheme \eqref{eq:TE} converges to the solution of the SDE \eqref{eq:SDE} in $L^p$. Precise statements are made in Appendix \autoref{strongconvergence}.

\begin{proposition}\label{trun_conv}
Under appropriate assumptions (more precisely, let \autoref{ass:convergence} in Appendix \autoref{strongconvergence} hold for $p=2$ and some even number $p_0>2$ sufficiently large), the projected schemes \eqref{scm21} converges in $L^2$ with rate $1/2$ for $r<1/(2(\k-1))$.
\end{proposition}
This is proved in Appendix \autoref{proof_trun_conv}. We need to point out that, with almost the same level of highest bounded moments required, the projected scheme \eqref{scm21} preserves the strong convergence rate $1/2$, while the balanced scheme \eqref{scm11} only gives $L^p$ convergence rate of $\a$, which in many cases is only $1/4$ - see, e.g., the result in \cite{Sabanis2014}.

\begin{corollary}
Given tamed coefficients $b^h,\s^h$ with which the scheme \eqref{eq:TE} converges to the solution of SDE \eqref{eq:SDE} in $L^2$, the composed scheme
\begin{equation}
\bx_{k+1}=\Pi\left(\bx_k+b^h(t_k,\bx_k)h+\s^h(t_k,\bx_k)\D W_{k+1}\right),
\end{equation}
with appropriate $r$ chosen, also converges in $L^2$ with the same rate.
\end{corollary}

\section{Asymptotic Stability of Equilibrium}\label{sec:stab}
\par
Suppose there is a unique solution to the SDE
\begin{equation}\label{sde_stab}
\td X_t=b(t,X_t)\td t+\sigma(t,X_t)\td W_t,~t\geqslant0,
\end{equation}
with drift and diffusion satisfying
\begin{equation}\label{coeff_stab}
b(t,0)=\sigma(t,0)\equiv0,~\forall t\geqslant0.
\end{equation}
In the context of stability one still needs to model the growths of $b$ and $\s$ in terms of the selected Lyapunov function in the class $V^p_\g$. But instead of $1+V$ as in the integrability discussion before, we need a different assumption than \autoref{as:poly} to model the growth conditions of $b$ and $\s$, due to \eqref{coeff_stab} and the possibility of $V$ taking the form \eqref{different_V}. More precisely,
\begin{assumption}\label{growth}
For each $V\in\mathcal{V}^p_\g$ there is a non-negative function $U\in\mathcal{C}(\RR^d),~\ker(U)=\{0\}$, s.t. $V(\cdot)\le U(\cdot)$, and that $\exists K>0,~\k_{1,2}\ge1$ s.t.
\begin{equation*}
|b(t,x)|\leqslant KU(x)^{\kappa_1\g},~\|\sigma(t,x)\|\leqslant KU(x)^{\kappa_2\g},~\forall t\geqslant0,~x\in\mathbb{R}^d.
\end{equation*}
\end{assumption}
In most cases the function $U$ can be reasonably assumed to have polynomial growth in the sense
\begin{equation*}
U(\cdot)\lesssim|\cdot|^{q_1}+|\cdot|^{q_2},
\end{equation*}
with $0<q_1\le q_2$, which gives polynomial growth for $b$ and $\s$ - see \autoref{lorenz}.

\begin{definition}
The SDE \eqref{sde_stab} is almost surely stable, if $X_t\to0$ a.s. as $t\to\infty$, regardless of the value of $X_0$. A numerical scheme $\{\bx_k\}$ of the SDE \eqref{sde_stab} is almost surely stable, if for fixed step size $h>0,~\bx_k\to0$ a.s. as $k\to\infty$, regardless of the value of $X_0$.
\end{definition}

\begin{definition}
Let $V\in\mathcal{V}^p_\g$. The SDE \eqref{sde_stab} is $V$-exponentially stable with rate $\rho$, if $\exists\rho>0$ s.t.
\begin{equation*}
\ex V(X_t)\le e^{-\rho t}\ex V(X_0),~\forall t\ge0.
\end{equation*}
A numerical scheme $\{\bx_k\}$ of the SDE \eqref{sde_stab} is $V$-exponentially stable with rate $\tro$, if for fixed time-step $h>0,~\exists\tro>0$ s.t.
\begin{equation*}
\ex V(\bx_k)\le e^{-\tro kh}\ex V(X_0),~\forall k\ge0.
\end{equation*}
\end{definition}

\begin{remark}
By Fatou's lemma, $V$-exponential stability implies almost-sure stability.
\end{remark}

First we check the conditions for stability of equilibrium on the SDE level. We first quote a simplified version of stochastic LaSalle theorem regarding the almost-sure stability of SDE \eqref{sde_stab} from \cite{Szpruch2010monotone,mao1999stochastic,shen2006improved}:

\begin{theorem}\label{sde_stability}
Suppose $b$ and $\sigma$ are locally Lipschitz in $x$ and let $V\in\mathcal{C}^2(\mathbb{R}^d)$ be non-negative. If there is a non-negative function $z\in\mathcal{C}(\mathbb{R}^d)$ s.t.
\begin{equation}
\mathcal{L}V(x)\leqslant-z(x),~\forall x\in\mathbb{R}^d,\label{LV}
\end{equation}
then almost surely we have
\begin{equation}
\overline{\lim_{t\to\infty}}V(X_t)<\infty,~\lim_{t\to\infty}z(X_t)=0,
\end{equation}
regardless of the value of $X_0$. In addition, if $\ker(z)=\{0\}$, then $X_t\to0$ a.s. as $t\to\infty$.

Moreover, as a special case when $z(\cdot)\ge\rho V(\cdot)$ for some positive constant $\rho$, then the trivial solution is $V$-exponentially stable.
\end{theorem}

One can use \autoref{sde_stability} to determine whether a system is almost surely stable. In particular, mean-square stability, i.e. $V(\cdot)=|\cdot|^2$, is the most popular choice. Before introducing the stability results for tamed Euler scheme let us consider the following simple case.
\begin{example}\label{ex_power4}
The trivial solution to
\begin{equation}\label{power4}
\td X_t=-|X_t|^2X_t\td t+|X_t|^2\td W_t,~|X_0|^2<\infty~\textrm{a.s.}
\end{equation}
is almost surely stable.
\end{example}
Indeed we and calculate
\begin{equation*}
\mathcal{L}|x|^2=-2|x|^4+|x|^4=-|x|^4=:-z(x),
\end{equation*}
where $z(x)\geqslant0$ and $z(x)=0\Leftrightarrow x=0$. Note that in this case the trivial solution is not mean-square exponentially stable, but \autoref{sde_stability} still holds.

Nevertheless, the stability property of the numerical scheme of \eqref{power4} is not immediate. We may as well investigate the balanced schemes \eqref{scm11}, where
\begin{equation}
b^h(x)=\frac{b(x)}{1+G(x)h^\alpha},~\sigma^h(x)=\frac{\sigma(x)}{1+G(x)h^\alpha},~0<\alpha\leqslant1.
\end{equation}
One first notices that by \eqref{square},
\begin{equation*}
|\bx_{k+1}|^2=|\bx_k|^2+\mathcal{L}^h|\bx_k|^2h+|b^h(\bx_k)|^2h^2+M_{k+1},
\end{equation*}
where $M_{k+1}$ denotes a local martingale. We then calculate
\begin{align*}
\mathcal{L}^h|x|^2=&2\frac{x\cdot b(x)}{1+G(x)h^\alpha}+\frac{\|\sigma(x)\|^2}{(1+G(x)h^\alpha)^2}\leqslant\frac{1}{1+G(x)h^\alpha}\mathcal{L}|x|^2\\
=&-\frac{z(x)}{1+G(x)h^\alpha}.
\end{align*}
One can choose $\alpha\leqslant1$ and $G(x):=2|x|^2$, s.t.
\begin{align*}
A^h(x):=&\frac{z(x)}{1+G(x)h^\alpha}-\frac{|b(x)|^2h}{(1+G(x)h^\alpha)^2}=\frac{|x|^4}{1+2|x|^2h^\alpha}-\frac{|x|^6h}{(1+2|x|^2h^\alpha)^2}\\
\geqslant&\frac{2|x|^6h^\alpha-|x|^6h}{(1+2|x|^2h^\alpha)^2}\geqslant\frac{|x|^6h}{(1+2|x|^2h^\alpha)^2}\geqslant0,
\end{align*}
and $A^h(x)=0\Leftrightarrow x=0$. Thus we have
\begin{equation*}
|\bx_{k+1}|^2\leqslant|\bx_k|^2-A^h(\bx_k)h+M_{k+1}\leqslant|\bx_0|^2-\sum_{l=0}^kA^h(\bx_l)h+\sum_{l=0}^kM_{l+1},
\end{equation*}
from which we deduce $\lim_{l\to\infty}A^h(\bx_l)=0$ a.s. $\Rightarrow \lim_{k\to\infty}\bx_k=0$ a.s. This is due to the following lemma quoted from \cite{mao2007stochastic} (Theorem 1.3.9):
\begin{lemma}\label{doob}
Consider a non-negative stochastic process $\{V_k\}$ with representation
\begin{equation*}
V_k=V_0+A^1_k-A^2_k+M_k,
\end{equation*}
where $\{A^1_k\}$ and $\{A^2_k\}$ are almost surely non-decreasing, predictable processes with $A^1_0=A^2_0=0$, and $\{M_k\}$ is a local martingale adapted to $\{\mathcal{F}_{t_k}\}$ with $M_0=0$. Then
\begin{equation}
\{\lim_{k\to\infty}A^1_k<\infty\}\subset\{\lim_{k\to\infty}A^2_k<\infty\}\cap\{\lim_{k\to\infty} V_k<\infty~\textrm{exsits}\}~\textrm{a.s.}
\end{equation}
\end{lemma}
This is in fact a discrete version of Theorem 2.6.7 in \cite{liptser1989theory} for special semimartingales.

Now we investigate if the tamed Euler scheme
\begin{equation}
\bx_{k+1}=\bx_k+b^h(t_k,\bx_k)h+\sigma^h(t_k,\bx_k)\Delta W_{k+1}\label{scheme_general}
\end{equation}
preserves (almost-sure) stability in the general case.

\begin{theorem}\label{stability}
Let $V\in\hat{\mathcal{V}}^p_\gamma:=\mathcal{V}^p_\g\cap\{V^{(p+1)}\equiv0\}$ be donimated by a function $U$. Suppose there is $z^h\in\mathcal{C}(\mathbb{R}^d),~z^h(\cdot)\geqslant0$, s.t.
\begin{equation}
\mathcal{L}^hV(x)\leqslant-z^h(x),\label{Lh}
\end{equation}
and a constant $0<\mu\le1$ s.t. $\forall t\ge0,~\forall x\in\mathbb{R}^d$,
\begin{equation}\label{taming_cond1}
\levv b^h(t,x)\revv h^{1/2}\vee\lev\s^h(t,x)\rev h^{1/4}\le\mu\frac{(1+U(x))^\g z^h(x)}{1+U(x)+z^h(x)},
\end{equation}
then for $\mu<1/\sqrt{c/2+cd^{p-1}(p-2)}$, the scheme (\ref{scheme_general}) satisfies:
\begin{equation*}
\overline{\lim_{k\to\infty}}V(\bx_k)<\infty,~\lim_{k\to\infty}z^h(\bx_k)=0,~\textrm{a.s.},
\end{equation*}
which further implies $\lim_{k\to\infty}\bx_k=0$ a.s. if $\ker(z^h)=\{0\}$.

Moreover, in the particular case where $z^h(\cdot)\ge\rho V(\cdot)$ for some $\rho>0$, if $\exists\mu>0$ s.t. $\forall t\ge0,~\forall x\in\mathbb{R}^d$,
\begin{equation}\label{taming_cond2}
\levv b^h(t,x)\revv h^{1/2}\vee\lev\s^h(t,x)\rev h^{1/4}\le\mu V(x)^\g,
\end{equation}
then the scheme (\ref{scheme_general}), with $\mu<\sqrt{\rho}/\sqrt{c/2+cd^{p-1}(p-2)}$, admits $V$-exponential stability with a rate $\hat{\rho}=\hat{\rho}(\mu)\in(0,\rho)$.
\end{theorem}
\begin{proof}
The proof is almost identical to that of \autoref{th:main}. However, instead of \eqref{taylor} we have, by the estimate for the remainder \eqref{Rh},
\begin{align*}
V(\bx_{k+1})=&V(\bx_k)+\mathcal{L}^hV(\bx_k)h+R^hV(\bx_k)+M_{k+1}\\
\le&V(\bx_k)-\mathcal{L}^hV(\bx_k)h+\frac{1}{2}\lev V^{(2)}(\bx_k)\rev\levv\bb_k\revv^2h^2\\
&+\sum_{\substack{3\le i+2j\le p\\i,j\in\NN}}\phi_{i+2j}\lev V^{(i+2j)}(\bx_k)\rev\levv\bb_k\revv^i\lev\bsig_k\rev^{2j}h^{i+2j}+M_{k+1}
\end{align*}
where $M_{k+1}$ is a local martingale. Notice that all derivatives of $V$ have upper bounds as defined in \eqref{V}. Now apply \eqref{Lh} and \eqref{taming_cond1} and we get
\begin{align*}
V&(\bx_{k+1})\le V(\bx_k)-z^h(\bx_k)h+\frac{1}{2}c\mu^2(1+V(\bx_k))^{1-2\g}\left(\frac{(1+U(\bx_k))^\g z^h(\bx_k)}{1+U(\bx_k)+z^h(\bx_k)}\right)^2h\\
&\qquad\qquad+\sum_{\substack{3\le i+2j\le p\\i,j\in\NN}}\phi_{i+2j}c\mu^{i+2j}(1+V(\bx_k))^{1-(i+2j)\g}\left(\frac{(1+U(\bx_k))^\g z^h(\bx_k)}{1+U(\bx_k)+z^h(\bx_k)}\right)^{i+2j}h^\frac{i+j}{2}+M_{k+1}\\
\le&V(\bx_k)-z^h(\bx_k)h+\frac{1}{2}c\mu^2\frac{1+U(\bx_k)}{\left(1+\frac{1+U(\bx_k)}{z^h(\bx_k)}\right)^2}h+\sum_{\substack{3\le i+2j\le p\\i,j\in\NN}}\phi_{i+2j}c\mu^{i+2j}\frac{1+U(\bx_k)}{\left(1+\frac{1+U(\bx_k)}{z^h(\bx_k)}\right)^{i+2j}}h^\frac{i+j}{2}+M_{k+1}\\
\le&V(\bx_k)-z^h(\bx_k)h+\frac{1}{2}c\mu^2\frac{1+U(\bx_k)}{1+\frac{1+U(\bx_k)}{z^h(\bx_k)}}h+\sum_{\substack{3\le i+2j\le p\\i,j\in\NN}}\phi_{i+2j}c\mu^{i+2j}\frac{1+U(\bx_k)}{1+\frac{1+U(\bx_k)}{z^h(\bx_k)}}h^\frac{i+j}{2}+M_{k+1}\\
\le&V(\bx_k)-z^h(\bx_k)h+\frac{1}{2}c\mu^2z^h(\bx_k)h+\sum_{\substack{3\le i+2j\le p\\i,j\in\NN}}\phi_{i+2j}c\mu^{i+2j}z^h(\bx_k)h^\frac{i+j}{2}+M_{k+1}\\
\le&V(\bx_k)-z^h(\bx_k)h+\frac{1}{2}c\mu^2z^h(\bx_k)h+\sum_{3\le s\le p}\phi_sc\mu^sz^h(\bx_k)h+M_{k+1}.\num\label{A2}
\end{align*}
This implies that
\begin{align}
V(\bx_{k+1})\leqslant&V(X_0)-\sum_{l=0}^k\left(1-\frac{1}{2}c\mu^2-\sum_{s=3}^pc\mu^s\phi_s\right)z^h(\bx_l)h+\sum_{l=0}^kM_{l+1}\label{A1}\\
=:&V(X_0)-A_k^2+\hat{M}_k.\nonumber
\end{align}
One should then find a taming method with $\mu$ sufficiently small s.t.
\begin{equation*}
1-\frac{1}{2}c\mu^2-\sum_{s=3}^pc\mu^s\phi_s>0,
\end{equation*}
so that $A^2_k$ is increasing in $k$. Now with $A^1_k\equiv0$ and $\hat{M}_k$ a local martingale with $\hat{M}_0=0$, according to \autoref{doob}, $\lim_{k\to\infty}V(\bx_k)<\infty$, and also
\begin{equation}
\sum_{l=0}^\infty\left(1-\frac{1}{2}c\mu^2-\sum_{s=3}^pc\mu^s\phi_s\right)z^h(\bx_l)h<\infty,~\textrm{a.s.}
\end{equation}
which implies that $\lim_{k\to\infty}z^h(\bx_k)=0$ a.s. Moreover when $z^h(x)=0$ iff $x=0$ one concludes that $\lim_{k\to\infty}\bx_k=0$ a.s. In fact, assuming $\mu\le1$, by \autoref{rmk:const} one just needs to choose $\mu<1/\sqrt{c/2+cd^{p-1}(p-2)}$.

If $z^h(\cdot)\ge\rho V(\cdot)$ for some $\rho>0$, one runs the same calculation to get, instead of (\ref{A1}),
\begin{align*}
V(\bx_{k+1})\le&V(\bx_k)-\left(\rho-\frac{1}{2}c\mu^2-\sum_{s=3}^pc\mu^s\phi_s\right)V(\bx_k)h+M_{k+1}\\
=:&V(\bx_k)-\tro V(\bx_k)h+M_{k+1}.
\end{align*}
Choose $\mu$ sufficiently small s.t. $\tro>0$. Taking expectation on both sides, one arrives at
\begin{align*}
\ex V(\bx_{k+1})\leqslant&(1-\hat{\rho}h)\ex V(\bx_k)\\
\leqslant&(1-\hat{\rho}h)^{k+1}\ex V(X_0)\leqslant e^{-\hat{\rho}(k+1)h}\ex V(X_0)\to0,
\end{align*}
as $k\to\infty$. This can be done by choosing $\mu<\sqrt{\rho}/\sqrt{c/2+cd^{p-1}(p-2)}$.
\end{proof}

\begin{remark}
In analogy to \autoref{lem:main}, \autoref{stability} also holds for $V\in\mathcal{V}^p_\g$.
\end{remark}

\begin{remark}
According to the calculations \eqref{A2}, condition \eqref{taming_cond1} can be weakened to
\begin{equation}\label{combo}
\lev V^{(i+2j)}(x)\rev\levv b^h(t,x)\revv^i\lev\s^h(t,x)\rev^{2j}h^\frac{i+j}{2}\le\mu z^h(x),~\forall t\ge0,~x\in\RR^d,
\end{equation}
for $i=2,j=0$ and all $i,j\in\NN$ s.t. $3\le i+2j\le p$.
\end{remark}

\begin{remark}\label{cond_V_bar}
For $V\in\bar{\mathcal{V}}^p_\g$ condition \eqref{taming_cond1} can be simplified to
\begin{equation}\label{taming_cond3}
\levv b^h(t,x)\revv h^{1/2}\vee\lev\s^h(t,x)\rev h^{1/4}\le\mu\frac{U(x)^\g z^h(x)}{U(x)+z^h(x)},~\forall t\ge0,~x\in\RR^d,
\end{equation}
which also implies \eqref{combo} for $0<\mu\le1$.
\end{remark}
Notice that \eqref{taming_cond3} is reasonable since from \eqref{Lh} we have
\begin{align}
z^h(x)\le&\lev\nabla V(x)\rev\levv b^h(t,x)\revv+\frac{1}{2}\lev V^{(2)}(x)\rev\lev\s^h(t,x)\rev^2\nonumber\\
\le&KU(x)^{1+(\k_1-1)\g}+KU(x)^{1+2(\k_2-1)\g},\label{z_growth}
\end{align}
which ensures no singularity in the right-hand-side term in \eqref{taming_cond3}.

\subsection{Balanced Schemes}
Now one can try to find out whether a certain type of taming scheme can preserve stability. It turns out that the subspace $\bar{\mathcal{V}}^p_\g$ is the most frequently used resource when in sought of Lyapunov functions. Let us first investigate the following type of tamed schemes adopted by \cite{hutzenthaler2012strong,tretyakov2012fundamental,Sabanis2014}:
\begin{equation}
b^h(t,x)=\frac{b(t,x)}{1+G(x)h^\alpha},~\sigma^h(t,x)=\frac{\sigma(t,x)}{1+G(x)h^\alpha},\label{scm1}
\end{equation}
for some $G(\cdot)\geqslant0<\alpha\leqslant1$. Given the growth condition \eqref{z_growth}, which also holds for $z(\cdot)$, it turns out that by imposing some lower bounds on $z$ one can recover almost-sure stability for \eqref{scm1}.

\begin{proposition}\label{stab1}
Let \autoref{growth} hold for $V\in\mathcal{V}^p_\gamma$ s.t. the coefficients of the SDE \eqref{sde_stab} satisfy
\begin{equation}
\mathcal{L}V(x)\leqslant-z(x),~\forall t\geqslant0,~\forall x\in\mathbb{R}^d,
\end{equation}
for some $0\leqslant z\in\mathcal{C}(\mathbb{R}^d)$ satisfying
\begin{equation}\label{cond_z}
z(x)\geqslant\lambda(1+U(x))^{1-\g}\left(U(x)^{\k_1\g}\vee U(x)^{\k_2\g}\right),~\forall x\in\RR^d,
\end{equation}
for some $\lambda>0$. Then, by choosing $h<(\mu\lambda/K)^4$ and $G(x)=C(U(x)^{(\k_1-1)\g}\vee U(x)^{(\k_2-1)\g}),~C\ge1/\left(\mu/K-h^{1/4}/\lambda\right),~\alpha\le1/4$, the Euler scheme (\ref{scheme_general}) with tamed coefficients (\ref{scm1}) preserves almost-sure stability for the trivial solution, where $\mu$ satisfies the requirement in \autoref{stability}.
\end{proposition}
\begin{proof}
First one calculates
\begin{align*}
\mathcal{L}^hV(x)=&\nabla V(x)\cdot\frac{b(t,x)}{1+G(x)h^\alpha}+\frac{1}{2(1+G(x)h^\alpha)^2}\tr\left[\nabla^2V(x)\sigma\sigma^\top(t,x)\right]\\
\leqslant&\frac{1}{1+G(x)h^\alpha}\mathcal{L}|x|^2\leqslant-\frac{z(x)}{1+G(x)h^\alpha}=:-z^h(x),\num\label{zh_G}
\end{align*}
which satisfies $z^h(x)=0\Leftrightarrow x=0$. Now one only needs to select appropriate $G(\cdot)$ and $\alpha$ s.t. condition \eqref{taming_cond1} is satisfied, i.e.,
\begin{align*}
\frac{|b(t,x)|h^\frac{1}{2}\vee\|\sigma(t,x)\|h^\frac{1}{4}}{1+G(x)h^\alpha}\leqslant&\mu\frac{(1+U(x))^\gamma}{1+U(x)+\frac{z(x)}{1+G(x)h^\alpha}}\frac{z(x)}{1+G(x)h^\alpha}\\
\Leftrightarrow~|b(t,x)|h^\frac{1}{2}\vee\|\sigma(t,x)\|h^\frac{1}{4}\leqslant&\frac{\mu (1+U(x))^\gamma}{\frac{1+U(x)}{z(x)}+\frac{1}{1+G(x)h^\alpha}},
\end{align*}
where the left-hand-side above, by \autoref{growth}, has upper bound $K\left(U(x)^{\k_1\g}\vee U(x)^{\k_2\g}\right)h^{1/4}$, while the right-hand-side, by \eqref{cond_z}, has lower bound $\mu(1+U(x))^\gamma/\left(\frac{(1+U(x))^\gamma}{\lambda\left(U(x)^{\k_1\g}\vee U(x)^{\k_2\g}\right)}+\frac{1}{1+G(x)h^\alpha}\right)$.
Hence one can require
\begin{align*}
&K\left(U(x)^{\k_1\g}\vee U(x)^{\k_2\g}\right)h^{1/4}\leqslant\mu(1+U(x))^\gamma/\left(\frac{(1+U(x))^\gamma}{\lambda\left(U(x)^{\k_1\g}\vee U(x)^{\k_2\g}\right)}+\frac{1}{1+G(x)h^\alpha}\right)\\
&\Leftrightarrow~\mu(1+U(x))^\gamma\geqslant\frac{K}{\lambda}h^{1/4}(1+U(x))^\g+\frac{K\left(U(x)^{\k_1\g}\vee U(x)^{\k_2\g}\right)}{1+G(x)h^\alpha}h^{1/4}\\
&\Leftrightarrow~1+G(x)h^\alpha\geqslant\frac{K\left(U(x)^{\k_1\g}\vee U(x)^{\k_2\g}\right)}{(\mu-Kh^{1/4}/\lambda)(1+U(x))^\gamma}h^{1/4},
\end{align*}
where for fixed $\mu\le1$ we choose $h\leqslant h_0<(\mu\lambda/K)^4$. Thus by choosing $\alpha=1/4$ and $G(x):=C\left(U(x)^{(\k_1-1)\g}\vee U(x)^{(\k_2-1)\g}\right)$, the taming condition \eqref{taming_cond3} is satisfied for $\mu\ge K\left(1/C+h^{1/4}/\lambda\right)$. Hence by \autoref{cond_V_bar} and \autoref{stability}, the scheme \eqref{scm1} is almost surely stable when $C$ and $h$ are chosen sufficiently large and small, respectively.
\end{proof}

When $U(\cdot)=|\cdot|^{q_1}+|\cdot|^{q_2},~0<q_1\le q_2$, one sees $U(\cdot)^{\k_1\g}\vee U(\cdot)^{\k_2\g}=|\cdot|^{(\k_1\wedge\k_2)q_1\g}+|\cdot|^{(\k_1\vee\k_2)q_2\g}$.

\begin{corollary}
In the special case where $V(\cdot)=|\cdot|^p$ and $z(x)\gtrsim|x|^{\kappa_1+p-1}+|x|^{\k_2+p-1}$, one just needs to choose $\alpha=1/4$ and $G(x):=C(|x|^{\kappa_1-1}+|x|^{\kappa_2-1})$ with $C$ sufficiently large.
\end{corollary}

\subsection{Projected Schemes}
In general there is not evident clue that the balanced scheme \eqref{scm1} can preserve moment-exponential stability. However, one may project it onto a bounded range:
\begin{equation}\label{scm3}
\bx_{k+1}=\Pi\left(\bx_k+b^h(t_k,\bx_k)h+\s^h(t_k,\bx_k)\D W_{k+1}\right),
\end{equation}
where $\Pi:\RR^d\to\RR^d$ is such that $|\Pi x|=|x|\wedge h^{-r}$ for some $r>0,~\forall x\in\RR^d$, and $b^h,\s^h$ are as in \eqref{scm1}. By adopting this scheme one can immediately have $z^h$ in \eqref{zh_G} replaced by just $z$ itself (with scaling):
\begin{align*}
z^h(x)=&\frac{z(x)}{1+G(x)h^\a}=\frac{z(x)}{1+C|x|^{\k^\ast}h^\a}\\
\ge&\frac{z(x)}{1+Ch^{\a-rq\k^\ast}}\ge\frac{1}{1+C}z(x),~\forall x\in\RR^d,
\end{align*}
by choosing $r<\a/(q\k^\ast)$, where $G(\cdot)$ is, for instance as in \autoref{ex_power4}, chosen to be $C|\cdot|^{\k^\ast}$ for some $C,\k^\ast>0$. This motivates the idea that \eqref{scm3} can serve as a remedy to the shortcoming of the balanced scheme \eqref{scm1}. Indeed, when $z(\cdot)\ge\rho V(\cdot)$, for the balance scheme one has
\begin{equation*}
\mathcal{L}^hV(x)\le-\rho\frac{V(x)}{1+G(x)h^\a},
\end{equation*}
where one sees that $z^h(\cdot)\gtrsim V(\cdot)$ is violated due to the unboundedness of $G(\cdot)$. However, this can be avoided by using projection \eqref{scm3}.
\begin{proposition}
Let \autoref{growth} hold with $U(\cdot)\le\nu(1+|\cdot|^q),~\nu,q>0$ and $V$ satisfying
\begin{equation}\label{V_increase}
V(x)\le V(y),~\forall x,y\in\RR^d,~|x|\le|y|.
\end{equation}
Suppose $\exists\rho>0$ s.t.
\begin{equation*}
\mathcal{L}V(x)\le-\rho V(x),~\forall x\in\RR^d.
\end{equation*}
Then, with $G(x):=C(1+|x|^{(\ck-1)q\g}),~C\ge K\nu^{(\ck-1)\g}/\mu,~\alpha\le1/4,~r<\a/((\ck-1)q\g)$, the scheme \eqref{scm3} is $V$-exponentially stable, where $\ck=\k_1\vee\k_2$ and $\mu$ satisfies the requirement in \autoref{stability}.
\end{proposition}
\begin{proof}
Notice that by the same argument as in the proof of \autoref{project+tame}, we treat $\mathcal{L}^h_{b^h,\s^h}=\mathcal{L}_{b^h,\s^h}$, and $b^h,\s^h$ in \autoref{stability} are just as in \eqref{scm1}. We first verify condition \eqref{taming_cond2}:
\begin{align*}
&\frac{\levv b(t,x)\revv h^{1/2}\vee\lev\s(t,x)\rev h^{1/4}}{1+G(x)h^\a}\le\mu V(x)^\g\\
\Leftarrow~&K(U(x)^{\k_1\g}\vee U(x)^{\k_2\g})h^{1/4}\le\mu V(x)^\g G(x)h^\a,
\end{align*}
which is achieved by choosing $\a\le1/4,~G(x):=C(1+|x|^{(\ck-1)q\g}),~C\ge K\nu^{(\ck-1)\g}/\mu$, due to $V\le U$ and the polynomial growth of $U$, assuming $\nu\ge1$ without loss of generality. Also for $x\in\{x:~|x|\le h^{-r}\}$, we have $G(x)\le C+Ch^{-r(\ck-1)q\g}$, and thus
\begin{equation*}
\mathcal{L}^hV(x)\le-\frac{\rho}{1+G(x)h^\a}V(x)\le-\frac{1}{1+Ch^\a+Ch^{\a-r(\ck-1)q\g}}V(x)=:-\tro V(x),
\end{equation*}
for $\tro>0$ if we choose $r<\a/((\ck-1)q\g)$. Note that there is no restriction on the step size $h$.
\end{proof}

In fact, one can show that a projected standard Euler scheme - with the original drift and diffusion inside:
\begin{equation}\label{scm2}
\bx_{k+1}=\Pi\left(\bx_k+b(t_k,\bx_k)h+\s(t_k,\bx_k)\D W_{k+1}\right),
\end{equation}
is enough to inherit $V$-exponential stability under feasible conditions. This has been introduced earlier in \eqref{scm21}, which by \autoref{trun_conv} is well-defined.

\begin{proposition}\label{stab2}
Let \autoref{growth} hold with $U=V$ satisfying \eqref{V_increase} and $V(\cdot)\le\nu(1+|\cdot|^q)$ for some $\nu,q>0$. If $\exists\rho>0$ s.t.
\begin{equation}
\mathcal{L}V(x)\leqslant-\rho V(x),~\forall x\in\mathbb{R}^d,
\end{equation}
then with $r<1/(4(\ck-1)q\g),~h<(\mu/(2K\nu^{(\ck-1)\g}))^\be$, the tamed Euler scheme \eqref{scm2} preserves $V$-exponential stability, where $\be=1/4-r(\ck-1)q\g$ and $\mu$ satisfies the requirement in \autoref{stability}.
\end{proposition}
\begin{proof}
As shown in \eqref{rem1} condition \eqref{Lh} is redundant and one only needs to verify condition \eqref{taming_cond2} for $b$ and $\s$, i.e.
\begin{equation}\label{con_stab2}
\levv b(t,x)\revv h^{1/2}\vee\lev\s(t,x)\rev h^{1/4}\le\mu V(x)^\g,~\forall t,x.
\end{equation}
The left-hand-side term has upper bound $K\left(V(x)^{\k_1\g}h^{1/2}\right)\vee\left(V(x)^{\k_2\g}h^{1/4}\right)$, and for scheme \eqref{scm2} we know $|\bx_k|\le h^{-r}$. Since $V(\cdot)\le\nu\left(1+|\cdot|^q\right)$, one can require
\begin{align*}
&\mu V(x)^\g\ge KV(x)^\g\left(V(x)^{(\k_1-1)\g}h^{1/2}\right)\vee\left(V(x)^{(\k_2-1)\g}h^{1/4}\right)\\
\Leftarrow~&\mu\ge K\nu^{(\ck-1)\g}\left(1+|x|^{(\k_1-1)q\g}\right)h^{1/2}\vee\left(1+|x|^{(\k_2-1)q\g}\right)h^{1/4}\\
\Leftarrow~&\mu\ge2K\nu^{(\ck-1)\g}\left(h^{1/2-r(\k_1-1)q\g}\vee h^{1/4-r(\k_2-1)q\g}\right)\\
\Leftarrow~&\mu\ge2K\nu^{(\ck-1)\g}h^\beta.\num\label{beta}
\end{align*}
Note that one can immediately let inequality \eqref{beta} hold by choosing
\begin{equation}\label{r}
r<\frac{1}{2(\k_1-1)q\g}\wedge\frac{1}{4(\k_2-1)q\g},~h<h_0\le\left(\frac{\mu}{2K\nu^{(\ck-1)\g}}\right)^{1/\beta},
\end{equation}
for fixed $\mu$. Therefore, the scheme \eqref{scm2} preserves $V$-exponential stability when such $r$ is chosen and $h$ is sufficiently small.
\end{proof}
Moment exponential stability immediately follows when $V(\cdot)=U(\cdot)=|\cdot|^p,~q=p=1/\g$.

On the other hand, scheme \eqref{scm2}, as expected, also admits almost-sure stability given same conditions as for scheme \eqref{scm1}.

\begin{proposition}
Let \autoref{growth} hold with $V\in\mathcal{V}^p_\g$ satisfying \eqref{V_increase}. Suppose $\exists0\le z\in\mathcal{C}(\RR^d)$ satisfying \eqref{cond_z}, s.t.
\begin{equation*}
\mathcal{L}V(x)\leqslant-z(x),~\forall x\in\mathbb{R}^d.
\end{equation*}
If $\exists\nu,q>0$ s.t. $U(\cdot)\le\nu(1+|\cdot|^q)$, then with $r<(4(\ck-1)q\g)^{-1},~h<\left(\mu\lambda/(K+2\lambda K\nu^{(\ck-1)\g})\right)^{1/\beta}$, the scheme (\ref{scm2}) is almost-surely stable, where $\beta=1/4-r(\ck-1)q\g$ and $\mu$ satisfies the requirement in \autoref{stability}.
\end{proposition}
\begin{proof}
Again one only needs to check condition \eqref{taming_cond1} for $b$ and $\s$ for scheme \eqref{scm2}, which satisfies $|\bx_k|\leqslant h^{-r},~\forall k\geqslant1$, with $z^h(\cdot)=z(\cdot)$. Indeed for all $x$ (regardless of $X_0$ since we are only interested in the long-term behaviour),
\begin{equation*}
|b(t,x)|h^{1/2}\vee\|\sigma(t,x)\|h^{1/4}\leqslant\mu\frac{(1+U(x))^\gamma z(x)}{1+U(x)+z(x)},
\end{equation*}
where, the left-hand-side term above has upper bound $Kh^{1/4}\left(U(x)^{\k_1\g}\vee U(x)^{\k_2)\g}\right)$, and the right-hand-side term minimizes when $z(x)$ reaches its lower bound in \eqref{cond_z}. Thus, due to $|x|\le h^{-r}$, one can require
\begin{align*}
&Kh^{1/4}\left(U(x)^{\k_1\g}\vee U(x)^{\k_2)\g}\right)\leqslant\mu\frac{\lambda(1+U(x))^\g\left(U(x)^{\k_1\g}\vee U(x)^{\k_2)\g}\right)}{(1+U(x))^\g+\lambda\left(U(x)^{\k_1\g}\vee U(x)^{\k_2\g}\right)}\\
\Leftrightarrow~&Kh^{1/4}\left(U(x)^{\k_1\g}\vee U(x)^{\k_2\g}\right)\le\left(\mu-\frac{K}{\lambda}h^{1/4}\right)(1+U(x))^\g\\
\Leftarrow~&\nu^{(\ck-1)\g}Kh^{1/4}(1+|x|^{(\ck-1)q\g})\le\mu-\frac{K}{\lambda}h^{1/4}\\
\Leftarrow~&\left(\frac{K}{\lambda}+\nu^{(\ck-1)\g}K\right)h^{1/4}+\nu^{(\ck-1)\g}Kh^{1/4-r(\ck-1)q\g}\le\mu.
\end{align*}
Set $r<(4(\ck-1)q\g)^{-1}$ s.t. $\beta=1/4-r\ck q\g>0$. One can then choose $h<\left(\mu\lambda/(K+2\lambda K\nu^{(\ck-1)\g})\right)^{1/\beta}$, and hence almost-sure stability is achieved.
\end{proof}

In most cases $V(\cdot)=U(\cdot)=|\cdot|^p$ is chosen, then $q=p=1/\g$ and the conditions become quite simple.
\begin{corollary}
In the special case where $V(\cdot)=|\cdot|^p$ and $z(x)\gtrsim|x|^{\k_1+p-1}+|x|^{\k_2+p-1}$, one just needs to choose $r$ and $h$ sufficiently small.
\end{corollary}

\subsection{Other Examples}
\begin{example}\label{lorenz}
Consider the Stochastic Lorenz Equation \cite{hutzenthaler2012numerical} in $\RR^3$ driven by a $3$-d Wiener process:
\begin{equation}
b(x)=\left(\begin{matrix}
\alpha_1(x_2-x_1)\\ -\a_1x_1-x_2-x_1x_3\\ x_1x_2-\a_2x_3
\end{matrix}\right),~\sigma(x)=\left(\begin{matrix}
\beta_1x_1 & 0 & 0 \\ 0 & \beta_2x_2 & 0 \\ 0 & 0 & \beta_3x_3
\end{matrix}\right),
\end{equation}
where $2\alpha_1>\beta_1^2,~\beta_2^2<2,~2\alpha_2>\beta_3^2$.
\end{example}
One can immediately check for the Lyapunov function $V(\cdot)=|\cdot|^2\in\bar{\mathcal{V}}^2_{1/2}$:
\begin{equation*}
\mathcal{L}|x|^2=-(2\alpha_1-\beta_1^2)x_1^2-(2-\beta_2^2)x_2^2-(2\alpha_2-\beta_3^2)x_3^2\le-\rho|x|^2,
\end{equation*}
where $\rho:=(2\a_1-\be_1^2)\wedge(2-\be_2^2)\wedge(2\a_2-\be_3^2)$. According to \autoref{sde_stability} the system \eqref{lorenz} is mean-square stable for the equilibrium. One can thus choose taming method \eqref{scm2} to preserve mean-square stability for the tamed Euler scheme. One observes
\begin{align*}
|b(x)|=&\sqrt{\a_1^2(x_2-x_1)^2+(\a_1x_1+x_2+x_3)^2+(x_1x_2-\a_2x_3)^2}\le K(|x|+|x|^2),\\
\|\sigma(x)\|=&\sqrt{\be_1^2x_1^2+\be_2^2x_2^2+\be_3x_3^2}\le K|x|
\end{align*}
where $K=\sqrt{5\a_1^2+4\a_1+\a_2^2+4}\vee\sqrt{\be_1^2+\be_2^2+\be_3^2}$. Then one can choose $U(x)=|x|+|x|^2,~\k_1=2,~\k_2=1$ for \autoref{growth} to hold. Note that due to $p=2$ in this case, one only needs requirement on $b(t,x)$ in \eqref{con_stab2}. Hence according to \autoref{stab2}, one needs to choose $r<1/2$ and $h<(2K)^{-1/(1/2-r)}$ sufficiently small.

\begin{example}\label{van_der_pol}
Consider the following $2$-d SDE with drift and diffusion similar to the Stochastic Duffing-van der Pol Oscilator \cite{hutzenthaler2012numerical}:
\begin{equation}\label{sde_van_der_pol}
b(x)=\left(\begin{matrix}x_2-\alpha_1x_1\\-\alpha_2x_2-x_1^3\end{matrix}\right),~\sigma(x)=\left(\begin{matrix}
0 & 0 & 0 \\
0 & \beta x_2 & 0
\end{matrix}\right),
\end{equation}
where $\alpha_1>0,~2\alpha_2>\beta^2$.
\end{example}
In this case one can set the Lyapunov function to be
\begin{equation}
V(x)=x_1^4+2x_2^2,
\end{equation}
which is from a broader class $\hat{\mathcal{V}}^4_{1/4}$. Then one observes that
\begin{equation*}
\mathcal{L}V(x)=-4\a_1x_1^4-(4\alpha_2-2\beta^2)x_2^2\le-\rho V(x),
\end{equation*}
where $\rho:=4\wedge(4\a_2-2\be^2)$. According to \autoref{sde_stability}, the trivial solution of \eqref{sde_van_der_pol} is $V$-exponentially stable. Therefore we consider using the projected scheme \eqref{scm2}, for which all conditions regarding $(b^h,\s^h,z^h)$ are reduced to those of $(b,\s,z)$ on the set $\{x:~|x|\le h^{-r}\}$. In this $2$-d case one can, for example, define
\begin{equation}
\Pi\left(\begin{matrix} x_1 \\ x_2 \end{matrix}\right)=\frac{1}{\sqrt{2}}\left(\begin{matrix} -h^{-r}\vee x_1\wedge h^{-r} \\ -h^{-r}\vee x_2\wedge h^{-r} \end{matrix}\right),
\end{equation}
s.t. $|\Pi x|\le h^{-r}$. Hence in order to verify condition \eqref{taming_cond2}, one only needs to estimate for $|x_1|\vee|x_2|\le h^{-r}/\sqrt{2}$,
\begin{align*}
|b(x)|h^{1/2}=&\left((\a_2+1)|x_2|+\a_1|x_1|+|x_1|^3\right)h^{1/2}\le\frac{\a_2+1}{\sqrt[4]{2}}|x_2|^{1/2}h^{1/2-r/2}+\frac{\a_1+1}{2}|x_1|h^{1/2-2r}\\
\le&\frac{\a_1\vee\a_2+1}{2}h^{1/2-2r}(|x_1|+2|x_2|^{1/2})\le\mu V(x)^{1/4},\\
\|\s(x)\|h^{1/4}=&|\beta||x_2|h^{1/4}\le\frac{|\beta|}{\sqrt[4]{2}}h^{1/4-r/2}|x_2|^{1/2}\le\mu V(x)^{1/4},
\end{align*}
where we choose $r<1/4$ and $\mu:=\max\{4(\a_1\vee\a_2+1)h^{1/2-2r}/2,|\beta|h^{1/4-r/2}/\sqrt[4]{2}\}\le1$. Thus according to \autoref{stability}, the projected scheme \eqref{scm2} is exponentially stable in $V$ when $h$ is chosen sufficiently small.

\section{Non-Negativity And Comparison Preservation}\label{sec:comp}

Apart from integrability and stability, there are some other properties on the SDE level that can be preserved via taming. For example, some SDEs have solution only in a bounded region, and especially in $1$-d case two SDEs with the same diffusion can be compared, subject to some conditions.
\subsection{Non-Negativity}
Consider a linear SDE
\begin{equation}
\td X_t=\mu X_t\td t+\sigma X_t\td W_t,
\end{equation}
where $\mu$ and $\sigma$ are non-zero constants. One knows that the solution is
\begin{equation*}
X_t=X_0\exp\left\{\left(\mu-\sigma^2/2\right)t+\sigma W_t\right\}\geqslant0,~\textrm{a.s.}
\end{equation*}
if $X_0\geqslant0$ a.s. However this may not be the case for the standard Euler scheme
\begin{equation*}
\bx_{k+1}=(1+\mu h)\bx_k+\sigma\bx_k\Delta W_{k+1}.
\end{equation*}
More precisely, suppose that $\bx_k\geqslant0$ a.s., then for $\sigma>0$,
\begin{equation*}
\prb(\bx_{k+1}<0)=\prb\left(\Delta W_{k+1}<-\frac{1+\mu h}{\sigma}\right)>0;
\end{equation*}
the same applies for $\sigma<0$ due to the symmetry of Gaussian distribution. However, one can avoid this situation by simply truncating the Wiener process. For SDEs with super-linear growth coefficients a little bit more work is needed to preserve non-negativity.

Non-negativity of the SDE can be regarded as a corollary of the comparison theorem \autoref{sde_comparison}. However, it turns out that one can deduce non-negativity by a much weaker condition than that of the comparison theorem.
\begin{lemma}\label{lem_pos}
Given a $1$-d SDE
\begin{equation}
\td X_t=b(t,X_t)\td t+\sigma(t,X_t)\td W_t,~X_0\geqslant0,~\textrm{a.s.}\label{sde_pos}
\end{equation}
Suppose
\begin{description}[noitemsep,nolistsep]\vspace*{-3mm}
\item[i)] The solution to SDE (\ref{sde_pos}) exists and is unique.
\item[ii)] $\sigma(t,x)$ has polynomial growth in $x$ and $b(t,x)$ satisfies one-sided Lipschitz condition:
\begin{equation}
(x-y)(b(t,x)-b(t,y))\leqslant K|x-y|^2,\label{4.3}
\end{equation}
\item[iii)] $b(t,0)\geqslant0,~\sigma(t,0)\equiv0,~\forall t\geqslant0$.
\end{description}\vspace*{-3mm}
Then $X_t\geqslant0$ a.s. $\forall t$.
\end{lemma}
This is proved in Appendix \autoref{proof_of_positivity}.

Now consider a tamed Euler scheme of \eqref{sde_pos}:
\begin{equation}
	\hat{X}_{k+1}=\bar{X}_k+b^h(t_k,\hat{X}_k)h+\sigma^h(t_k,\hat{X}_k)\sqrt{h}\xi,\label{3.1}
\end{equation}
where $\xi\sim N(0,1)$. Non-negativity generally does not hold any more for $\hat{X}_k$, but one can recover this property by truncating the noise:
\begin{equation}\label{zeta}
	\zeta_h=\begin{cases}A_h,~\xi>A_h,\\ \xi,~-A_h\leqslant\xi\leqslant A_h,\\-A_h,~\xi<-A_h,\end{cases}
\end{equation}
where one takes $A_h=\sqrt{2|\log h|}$ to preserve strong convergence (See section 1.3.4 in \cite{milstein2004stochastic}).

\begin{theorem}
Let the assumptions in \autoref{lem_pos} hold. If one can find tamed method such that the coefficients in \eqref{3.1} satisfies
\begin{equation}
|\hat{b}^h(t,x)|h^\a\vee|\sigma^h(t,x)|h^{\a/2}\leqslant \mu|x|,~\forall t\ge0,~x\in\RR,\label{3.2}
\end{equation}
for some $\mu,\a>0$, where $\hat{b}(t,x)=b(t,x)-b(t,0)$, then the tamed Euler scheme
\begin{equation}\label{tame_pos}
\bx_{k+1}=\bx_k+b^h(t_k,\bx_k)h+\s^h(t_k,\bx_k)\sqrt{h}\zeta_h,
\end{equation}
is almost surely non-negative for $\a<1$ and $h$ sufficiently small.
\end{theorem}
\begin{proof}
Rewrite the scheme \eqref{tame_pos} as
\begin{align*}
\bx_{k+1}=&\bx_k+b(t_k,0)h+\hat{b}^h(t_k,\bx_k)h+\sigma^h(t_k,\bx_k))\sqrt{h}\zeta_h\\
\geqslant&\bx_k\left(1-\mu h^{1-\alpha}-\mu h^{1/2-\a/2}A_h\right)\num\label{3.3}
\end{align*}
as $b(t,0)\ge0$. In order for \eqref{3.3} to stay nonnegative, we must have $\a<1$ and
\begin{equation}
	h^{1-\a}+h^{1/2-\a/2}A_h\leqslant\frac{1}{\mu}.
\end{equation}
\end{proof}

As we already know \eqref{3.2} can be realised by letting
\begin{equation}
	\hat{b}^h(t,x)=\frac{\hat{b}(t,x)}{1+G(x)h^\alpha},~\s^h(t,x)=\frac{\s(t,x)}{1+G(x)h^\a},\label{3.4}
\end{equation}
for some $G(\cdot)\ge0,~0<\a<1$, to linearise $\hat{b}(t,x)$. If $\hat{b}(t,x)$, hence $b(t,x)$, has polynomial growth with degree $m$, one may think of letting $G(x)=C|x|^{m-1}$ with $C>0$, then
\begin{align*}
|\hat{b}^h(t,x)|=&\frac{|\hat{b}(t,x)|}{1+C|x|^{m-1}h^\alpha}\lesssim\frac{1+|x|^m}{h^\alpha(1+|x|^{m-1})}\\
\leqslant&\frac{1+|x|^{m-1}+|x|+|x|^m}{h^\alpha(1+|x|^{m-1})}=(1+|x|)h^{-\alpha}.
\end{align*}

However for the projected scheme \eqref{scm2}, one does not need to further truncate the noise via \eqref{zeta}. Instead one only needs to define a reasonable truncation $\Pi$, e.g. similar to what is suggested in \cite{Jean-FrancoisChassagneux2014},
\begin{equation}
\Pi x=\left(0\vee x_i\wedge h^{-r}\right)_{i=1,\cdots,d},
\end{equation}
where $r$ is chosen s.t. \autoref{trun_conv} holds.

\subsection{Comparison Result}
As an extension of non-negativity preservation, one can preserve comparison result for SDEs by applying taming techniques.
\par It is known that two SDEs with the same diffusion can be compared by the comparison theorem:
\begin{theorem}\label{sde_comparison}
Consider two SDEs:
\begin{align*}
\td X_t=&\nu(t,X_t)\td t+\sigma(t,X_t)\td W_t,\\
\td Y_t=&\lambda(t,X_t)\td t+\sigma(t,Y_t)\td W_t.
\end{align*}
Suppose the following conditions are satisfied:
\begin{description}[noitemsep,nolistsep]\vspace*{-3mm}
\item[(i)] $\nu,\lambda,\sigma$ are continuous in $x$.
\item[(ii)] $X_t$ and $Y_t$ exist and are unique, respectively.
\item[(iii)] $X_0\leqslant Y_0$ a.s.
\item[(iv)] $\nu(t,x)\leqslant\lambda(t,x),~\forall t\ge0, x\in\mathbb{R}$.
\item[(v)] Either $\lambda$ or $\mu$ satisfies one-sided Lipschitz condition \eqref{4.3}.
\end{description}\vspace*{-3mm}
Then $X_t\leqslant Y_t$ a.s. $\forall t\geqslant0$.
\end{theorem}
One can prove using the same recipe as the proof of \autoref{lem_pos}. The simplest case being $\nu(t,x)=\nu x,~\lambda(t,x)=\lambda x,~\sigma(t,x)=\sigma x$, where $\nu,\lambda,\sigma>0$ are constants, we know that $X_t$ and $Y_t$ are both non-negative if $X_0,Y_0\geqslant0$. Hence $X_t\leqslant Y_t$ a.s., $\forall t$, if $\nu\leqslant\lambda$ and $X_0\leqslant Y_0$.

Now consider the Euler scheme for each equation
\begin{align*}
\hat{X}_{k+1}=&\hat{X}_k+\nu(t_k,\hat{X}_k)h+\sigma(t_k,\hat{X}_k)\sqrt{h}\xi,\\
\hat{Y}_{k+1}=&\hat{Y}_k+\lambda(t_k,\hat{Y}_k)h+\sigma(t_k,\hat{Y}_k)\sqrt{h}\xi,
\end{align*}
where $\xi~N(0,1)$. In general comparison result does not necessarily hold for $\hat{X}_k$ and $\hat{Y}_k$, but by truncating the noise using \eqref{zeta} it can be recovered.

\begin{theorem}
Let the assumptions in \autoref{sde_comparison} hold with $\lambda$ satisfying one-sided Lipschitz condition \eqref{4.3}. If there is a taming method s.t.
\begin{equation}\label{tame_comp}
|\lambda^h(t,x)-\lambda^h(t,y)|h^\a\vee|\s^h(t,x)-\s^h(t,y)|h^{\a/2}\le \mu|x-y|,~\forall x,y\in\RR,~t\ge0,
\end{equation}
for some $\mu,\a>0$, and
\begin{equation}
\nu^h(t,x)\le\lambda^h(t,x),~\forall t\ge0,~x\in\RR,
\end{equation}
then the tamed Euler schemes
\begin{align*}
\bx_{k+1}=&\bx_k+\nu^h(t_k,\bx_k)h+\sigma^h(t_k,\bx_k)\sqrt{h}\zeta_h,\\
\by_{k+1}=&\by_k+\lambda^h(t_k,\by_k)h+\sigma^h(t_k,\by_k)\sqrt{h}\zeta_h,
\end{align*}
where $\zeta_h$ is defined as in \eqref{zeta}, preserve comparison result, i.e.,
\begin{equation*}
\bx_k\le\by_k,~\forall k\ge0,
\end{equation*}
for $\a<1$ and $h$ sufficiently small.
\end{theorem}
\begin{proof}
Write $\lambda(t,x)=\lambda(t,0)+\hat{\lambda}(t,x)$, then \eqref{tame_comp} implies $|\hat{\lambda}^h(t,x)-\hat{\lambda}^h(t,y)|\leqslant \mu|x-y|h^{-\alpha},~\forall t\ge0,~x,y\in\RR$. Thus
\begin{align*}
\by_{k+1}-\bx_{k+1}\geqslant&(\by_k-\bx_k)(1-\mu h^{1/2-\a/2}A_h)+(\lambda^h(\by_k)-\nu^h(\bx_k))h\\
\geqslant&(\by_k-\bx_k)(1-\mu h^{1/2-\a/2}A_h)+(\lambda^h(\by_k)-\lambda^h(\bx_k))h\\
\geqslant&(\by_k-\bx_k)(1-\mu h^{1-\alpha}-\mu h^{1/2-\a/2}A_h).
\end{align*}
Require $\a<1$ and $h^{1-\alpha}+h^{1/2-\a/2}A_h\leqslant1/\mu$, then comparison result still holds for $\bx_k$ and $\by_k$.
\end{proof}

Condition $\nu^h(t,x)\le\lambda^h(t,x)$ is usually immediately satisfied given $\nu(t,x)\le\lambda(t,x),~\forall t,x$. Now investigate whether \eqref{tame_comp} can be possible. Assume $\lambda(t,x)$ is differentiable for all $x$ and $|\partial_x\lambda(t,x)|\vee|\lambda(t,x)|\leqslant K(1+|x|^m)$ for some $K>0,m\ge1$. Consider
\begin{equation*}
\lambda^h(t,x)=\frac{\lambda(t,x)}{1+h^\alpha|x|^{m-1}},
\end{equation*}
By the mean-value theorem we have $|\lambda^h(t,x)-\lambda^h(t,y)|\le|\partial_x\lambda^h(t,\xi)||x-y|$ for some $\xi$ between $x$ and $y$. Then one applies the chain rule,
\begin{align*}
|\partial_x\lambda^h(t,\xi)|\leqslant&\frac{|\partial_x\lambda^h(t,\xi)|(1+h^\alpha|\xi|^{m-1})+|\lambda(t,\xi)|h^\alpha(m-1)|\xi|^{m-2}}{(1+h^\alpha|\xi|^{m-1})^2}\\
\leqslant&K\frac{(1+|\xi|^{m-1})(1+|\xi|^{m-1}h^\alpha)+(1+|\xi|^m)|\xi|^{m-2}h^\alpha}{(1+h^\alpha|\xi|^{m-1})^2}\\
=&K\frac{1+h^\alpha|\xi|^{m-2}+(1+h^\alpha)|\xi|^{m-1}+2h^\alpha|\xi|^{2m-2}}{1+2h^\alpha|\xi|^{m-1}+h^{2\alpha}|\xi|^{2m-2}}\\
\leqslant&K\frac{1+2|\xi|^{m-1}+h^\alpha|\xi|^{2m-2}}{h^\alpha(1+2|\xi|^{m-1}+h^\alpha|\xi|^{2m-2})}\\
=&Kh^{-\alpha},
\end{align*}
which implies that $|\lambda^h(x)-\lambda^h(y)|\leqslant K|x-y|h^{-\alpha}$.

\begin{appendices}
\section{$V$-Integrability Applied to Strong Convergence}\label{strongconvergence}
In the context of strong convergence, one needs the following setting for the SDE
\begin{equation} \label{eq:SDE_app}
\td X_t = b(t,X_t)\td t+\s(t,X_t)\td W_t,~t\in[0,T].
\end{equation}
\begin{assumption}\label{ass:convergence}
For a given number $p\ge1$, there is an even number $p_0>p$ sufficiently large (depending on the choice of taming) s.t. the coefficients of SDE \eqref{eq:SDE_app} satisfy, $\forall t,s\in[0,T],~x,y\in\RR^d$,
\begin{itemize}\vspace*{-5mm}
\item[i)]$\lj x-y,b(t,x)-b(t,y)\rj+\frac{p_0-1}{2}\lev\sigma(t,x)-\sigma(t,y)\rev^2\lesssim|x-y|^2$;
\item[ii)]$|b(t,0)|\vee\|\s(t,0)\|\vee\ex|X_0|^{p_0}<\infty$;
\item[iii)]$|b(t,x)-b(t,y)|\lesssim \left(1+|x|^{\k-1}+|y|^{\k-1}\right)|x-y|$ and\\
$\|\s(t,x)-\s(t,y)\|\lesssim \left(1+|x|^{(\k-1)/2}+|y|^{(\k-1)/2}\right)|x-y|$, for some $\k\ge1$;
\item[iv)]$|b(t,x)-b(s,x)|\lesssim \left(1+|x|^\k\right)|t-s|$ and\\
$\|\s(t,x)-\s(s,x)\|\lesssim \left(1+|x|^{(\k+1)/2}\right)|t-s|$.
\end{itemize}\vspace*{-5mm}
\end{assumption}
Note that i) and ii) above imply that $\forall t\in[0,T],~x,y\in\RR^d$,
\begin{equation*}
\lj x-y,b(t,x)-b(t,y)\rj+\frac{p-1}{2}\lev\sigma(t,x)-\sigma(t,y)\rev^2\lesssim|x-y|^2,
\end{equation*}
which is needed for the one-step perturbation\footnote{Or one-step stability, not to be confused with the stability of equilibrium.} estimate $\left(X_{t,x}(t+h)-X_{t,y}(t+h)\right)$ for the SDE. If we let $V(\cdot)=|\cdot|^{p_0}\in\hat{\mathcal{V}}^{p_0}_{1/p_0}$, then i) and ii) also imply that
\begin{equation}\label{mono}
\mathcal{L}V(x)=|x|^{p_0-2}\left(\lj x,b(t,x)\rj+\frac{p_0-1}{2}\lev\s(t,x)\rev^2\right)\lesssim1+V(x),
\end{equation}
which together with the growth condition implied by ii) and iii),
\begin{equation}\label{growth3}
|b(t,x)|\lesssim1+|x|^\k,~\|\s(t,x)\|\lesssim1+|x|^{(\k+1)/2},~\forall t\in[0,T],~x\in\RR^d,
\end{equation}
can make it possible for the tamed Euler scheme to achieve \autoref{th:main}.

By comparing the one-step error against the standard Euler scheme
\begin{equation}\label{eq:SEuler}
\tx_{t,x}(t+h) = x + b(t,x)h + \s(t,x)(W_{t+h}-W_t),
\end{equation}
the result in \cite{tretyakov2012fundamental} can be concluded as follows:
\begin{theorem}\label{convergence}
Let \autoref{ass:convergence} hold for some $p\ge1$ and $p_0>p$ sufficiently large. If the one-step difference against the standard Euler scheme \eqref{eq:SEuler} satisfies
\begin{align*}
&\lev\bx_{t,x}(t+h)-\tx_{t,x}(t+h)\rev_{L^p(\Omega)}\lesssim\left(1+|x|^\a\right)h^\delta,\\
&\levv\ex\bx_{t,x}(t+h)-\ex\tx_{t,x}(t+h)\revv\lesssim\left(1+|x|^{\a'}\right)h^{\delta+1/2},
\end{align*}
for some $\a,\a'>0,~\delta>1/2$, then
\begin{equation*}
\lev\bx_k-X_{t_k}\rev_{L^p(\Omega)}=O(h^{\delta-1/2}).
\end{equation*}
\end{theorem}
Here \autoref{th:main} plays an essential role in controlling the highest ($p_0$) moments of $\{\bx_k\}$ needed for $L^p$ convergence, which depends on what specific taming method one adopts, and how one decomposes the global error. One can find out the $p_0$ with respect to the balanced schemes in \cite{hutzenthaler2012strong,tretyakov2012fundamental,Sabanis2014}.

\section{Proof of Proposition \ref{trun_conv}}\label{proof_trun_conv}
\begin{proof}
Since both drift and diffusion are Lipschitz in $t$, we may assume $b(t,x)=b(x),~\sigma(t,x)=\sigma(x),~\forall t,x$. Notice that using a more precise growth condition \eqref{growth3} rather than \autoref{as:poly}, we can estimate $|b|h^{1/2}$ and $\|\s\|h^{1/4}$ separately in \eqref{trun_bound} and need only choose $r<1/(2(\k-1))$, where $q\g=1$.

One only needs to check if $\delta=1$ in \autoref{convergence}. Indeed the weak one-step error has estimate, by Cauchy-Schwartz inequality and Chebyshev's inequality (denote $\D W:=W_{t+h}-W_t$),
\begin{align*}
\left|\ex\bx_{t,x}-\ex\tx_{t,x}\right|=&\left|\ex\Pi(x+b(x)h+\sigma(x)\Delta W)-\ex\left(x+b(x)h+\sigma(x)\Delta W\right)\right|\\
\leqslant&2\ex\left|x+b(x)h+\sigma(x)\Delta W\right|\one_{|x+b(x)h+\sigma(x)\Delta W|>h^{-r}}\\
\leqslant&K\left(\ex|x+b(x)h+\sigma(x)\Delta W|^{2+\frac{3}{r}}\right)^\frac{1}{2}h^\frac{3}{2}\\
\leqslant&K\left(|x|^{1+\frac{3}{2r}}+\left((1+|x|)h^\frac{1}{2}\right)^{1+\frac{3}{2r}}+\left((1+|x|)h^\frac{1}{4}\right)^{1+\frac{3}{2r}}\right)h^\frac{3}{2}\\
\leqslant&K\left(1+|x|^{1+\frac{3}{2r}}\right)h^\frac{3}{2},
\end{align*}
where we used \eqref{trun_bound} for $|x|\le h^{-r}$. Similarly,
\begin{align*}
\ex\left|\bx_{t,x}-\tx_{t,x}\right|^2=&\ex\left|\Pi(x+b(x)h+\sigma(x)\Delta W)-x-b(x)h-\sigma(x)\Delta W\right|^2\\
\leqslant&K\ex\left|x+b(x)h+\sigma(x)\Delta W\right|^2\one_{|x+b(x)h+\sigma(x)\Delta W|>h^{-r}}\\
\leqslant&K\left(\ex|x+b(x)h+\sigma(x)\Delta W|^{4+\frac{4}{r}}\right)^\frac{1}{2}h^2\\
\leqslant&K\left(|x|^{2+\frac{2}{r}}+\left((1+|x|)h^\frac{1}{2}\right)^{2+\frac{2}{r}}+\left((1+|x|)h^\frac{1}{4}\right)^{2+\frac{2}{r}}\right)h^2\\
\leqslant&K\left(1+|x|^{2+\frac{2}{r}}\right)h^2.
\end{align*}
This validates the $L^2$ convergence of \eqref{scm21}.
\end{proof}
It is worth mentioning that here about $4\k$ moments are needed for the scheme, which is almost the same number of moments required for the balanced scheme on one step. However, as shown in Lemma 3.1 in \cite{tretyakov2012fundamental}, there is a further increase in the necessary number of bounded moments, implying that one needs less $p_0$ in \autoref{ass:convergence} for the projected scheme \eqref{scm21} than the balanced scheme \eqref{scm11}.

\section{Proof of Lemma \ref{lem_pos}}\label{proof_of_positivity}
\begin{proof}
Consider $f(x)=x^-=\max(0,-x)$. Take a sequence of smooth functions $\phi_n(x)\in\mathcal{C}^2(\RR)$ s.t.
\begin{equation*}
\phi_n(x)\to f(x),~\phi_n'(x)\to-\one_{\{x<0\}}(x),~\phi_n''(x)\to0,
\end{equation*}
uniformly as $n\to\infty$. For example one can choose $\phi_n(x)\lesssim1/(n|x|^m)$, if we assume $|\s(t,x)|^2\le K|x|^m$ for some $K>0,~m\ge1$. Apply It\^{o}'s formula to $\phi_n(X_t)$ we get
\begin{align*}
\td\phi_n(X_t)=&\phi_n'(X_t)\td X_t+\frac{1}{2}\phi_n''(X_t)\td\lj X\rj_t\\
=&\phi_n'(X_t)(b(t,X_t)\td t+\sigma(t,X_t)\td W_t)+\frac{1}{2}\phi_n''(X_t)\sigma^2(t,X_t)\td t.
\end{align*}
Let $n\to\infty$ we get
\begin{equation}\label{ito_pos}
X_t^-=\int_0^t-\one_{\{X_s<0\}}(b(s,X_s)\td s+\sigma(s,X_s)\td W_s).
\end{equation}
From \eqref{4.3} one can show that $b(t,x)=b_1(t,x)+b_2(t,x)$, where $b_1(t,x)$ is monotonically decreasing in $x$, and $b_2(t,x)$ is Lipschitz. One can, e.g., choose $b_2(t,x)=Kx$ and hence
\begin{align*}
(x-y)(b_1(t,x)-b_1(t,y))=&(x-y)(b(t,x)-Kx-b(t,y)+Ky)\\
=&(x-y)(b(t,x)-b(x,y))-K|x-y|^2\leqslant0.
\end{align*}
Thus, by taking expectation on both sides of \eqref{ito_pos}, we get
\begin{align*}
\ex X_t^-=&\ex\int_0^t-\one_{\{X_s<0\}}b(s,X_s)\td s\\
=&\ex\int_0^t-\one_{\{X_s<0\}}\left(b_1(s,X_s)+b_2(s,X_s)\right)\td s\\
\leqslant&\ex\int_0^t-\one_{\{X_s<0\}}\left(b_1(s,0)+b_2(s,0)-K|X_s|\right)\td s\\
=&\ex\int_0^t\one_{\{X_s<0\}}\left(-b(x,0)+K|X_s|\right)\td s.
\end{align*}
Note that $b(s,0)\geqslant0$, thus
\begin{equation*}
\ex X_t^-\leqslant\int_0^tK\ex X_s^-\td s~\Rightarrow~\ex X_t^-=0,~\forall t,
\end{equation*}
by Gronwall's inequality, which furthermore implies that $X_t\geqslant0$ a.s.
\end{proof}
\end{appendices}

\newpage
\bibliographystyle{plain}
\bibliography{v_stab}

\begin{thebibliography}{10}

\bibitem{Jean-FrancoisChassagneux2014}
Jean-Francois Chassagneux, Antoine Jacquier, and Ivo Mihaylov.
\newblock An explicit euler scheme with strong rate of convergence for
  non-lipschitz sdes.
\newblock {\em arXiv:1405.3561}, 2014.

\bibitem{MR3082312}
Desmond~J. Higham, Xuerong Mao, and Lukasz Szpruch.
\newblock Convergence, non-negativity and stability of a new {M}ilstein scheme
  with applications to finance.
\newblock {\em Discrete Contin. Dyn. Syst. Ser. B}, 18(8):2083--2100, 2013.

\bibitem{higham2000stability}
D.J. Higham.
\newblock A-stability and stochastic mean-square stability.
\newblock {\em BIT Numerical Mathematics}, 40(2):404--409, 2000.

\bibitem{higham2001mean}
D.J. Higham.
\newblock {Mean-square and asymptotic stability of the stochastic theta
  method}.
\newblock {\em SIAM Journal on Numerical Analysis}, 38:753--769, 2001.

\bibitem{higham2003exponential}
D.J. Higham, X.~Mao, and A.M. Stuart.
\newblock {Exponential mean-square stability of numerical solutions to
  stochastic differential equations}.
\newblock {\em LMS J. Comput. Math}, 6:297--313, 2003.

\bibitem{higham2003strong}
D.J. Higham, X.~Mao, and A.M. Stuart.
\newblock {Strong convergence of Euler-type methods for nonlinear stochastic
  differential equations}.
\newblock {\em SIAM Journal on Numerical Analysis}, 40(3):1041--1063, 2003.

\bibitem{higham2008almost}
D.J. Higham, X.~Mao, and C.~Yuan.
\newblock {Almost sure and moment exponential stability in the numerical
  simulation of stochastic differential equations}.
\newblock {\em SIAM Journal on Numerical Analysis}, 45(2):592--609, 2008.

\bibitem{hutzenthaler2011strong}
M.~Hutzenthaler, A.~Jentzen, and P.E. Kloeden.
\newblock Strong and weak divergence in finite time of {E}uler's method for
  stochastic differential equations with non-globally {L}ipschitz continuous
  coefficients.
\newblock {\em Proceedings of the Royal Society A}, 467(2130):1563--1576, 2011.

\bibitem{hutzenthaler2012strong}
M.~Hutzenthaler, A.~Jentzen, and P.E. Kloeden.
\newblock {Strong convergence of an explicit numerical method for SDEs with
  nonglobally Lipschitz continuous coefficients}.
\newblock {\em The Annals of Applied Probability}, 22(4):1611--1641, 2012.

\bibitem{hutzenthaler2012numerical}
Martin Hutzenthaler and Arnulf Jentzen.
\newblock Numerical approximations of stochastic differential equations with
  non-globally lipschitz continuous coefficients.
\newblock {\em arXiv preprint arXiv:1203.5809}, 2012.

\bibitem{Hutzenthaler2014}
Martin Hutzenthaler, Arnulf Jentzen, and Xiaojie Wang.
\newblock Exponential integrability properties of numerical approximation
  processes for nonlinear stochastic differential equations.
\newblock {\em arXiv:1309.7657}, 2014.

\bibitem{chas1980stochastic}
R.Z. Khasminski.
\newblock {\em {Stochastic Stability of Differential Equations}}.
\newblock Kluwer Academic Pub, 1980.

\bibitem{kloeden1992numerical}
P.E. Kloeden and E.~Platen.
\newblock {\em {Numerical Solution of Stochastic Differential Equations}}.
\newblock Springer, 1992.

\bibitem{liptser1989theory}
R.S. Liptser and A.N. Shiryayev.
\newblock {\em {Theory of Martingales}}.
\newblock Kluwer Academic Publishers, 1989.

\bibitem{mao1991stability}
X.~Mao.
\newblock {\em Stability of stochastic differential equations with respect to
  semimartingales}.
\newblock Longman Scientific \& Technical, 1991.

\bibitem{mao1999stochastic}
X.~Mao.
\newblock {Stochastic versions of the LaSalle theorem}.
\newblock {\em Journal of Differential Equations}, 153(1):175--195, 1999.

\bibitem{mao2007stochastic}
X.~Mao.
\newblock {\em {Stochastic Differential Equations and Applications}}.
\newblock Horwood Pub Ltd, 2007.

\bibitem{szpruch-diss}
X.~Mao and L.~Szpruch.
\newblock {Strong convergence rates for backward Euler--Maruyama method for
  non-linear dissipative-type stochastic differential equations with
  super-linear diffusion coefficients}.
\newblock {\em Stochastics}, 85:144--177, 2012.

\bibitem{Szpruch2010monotone}
X.~Mao and L.~Szpruch.
\newblock {Strong convergence and stability of implicit numerical methods for
  stochastic differential equations with non-globally Lipschitz continuous
  coefficients}.
\newblock {\em J. Comput. Appl. Math.}, 238:14--28, 2013.

\bibitem{MR1931266}
J.~C. Mattingly, A.~M. Stuart, and D.~J. Higham.
\newblock Ergodicity for {SDE}s and approximations: locally {L}ipschitz vector
  fields and degenerate noise.
\newblock {\em Stochastic Process. Appl.}, 101(2):185--232, 2002.

\bibitem{milstein2004stochastic}
G.N. Milstein and M.V. Tretyakov.
\newblock {\em {Stochastic Numerics for Mathematical Physics. Scientific
  Computation}}.
\newblock Springer-Verlag, Berlin, 2004.

\bibitem{sabanis2013note}
Sotirios Sabanis.
\newblock A note on tamed euler approximations.
\newblock {\em Electronic Communications in Probability}, 18:1--10, 2013.

\bibitem{Sabanis2014}
Sotirios Sabanis.
\newblock Euler approximations with varying coefficients: The case of
  superlinearly growing diffusion coefficients.
\newblock {\em arXiv:1308.1796}, 2014.

\bibitem{shen2006improved}
Y.~Shen, Q.~Luo, and X.~Mao.
\newblock {The improved LaSalle-type theorems for stochastic functional
  differential equations}.
\newblock {\em Journal of Mathematical Analysis and Applications},
  318(1):134--154, 2006.

\bibitem{tretyakov2012fundamental}
MV~Tretyakov and Z~Zhang.
\newblock A fundamental mean-square convergence theorem for sdes with locally
  lipschitz coefficients and its applications.
\newblock {\em SIAM Journal on Numerical Analysis}, 51(6):3135--3162, 2012.

\bibitem{MR2658159}
Fuke Wu, Xuerong Mao, and Lukas Szpruch.
\newblock Almost sure exponential stability of numerical solutions for
  stochastic delay differential equations.
\newblock {\em Numer. Math.}, 115(4):681--697, 2010.

\end{thebibliography}

\end{document}